\documentclass[12pt, reqno]{amsart}
\usepackage{amsmath, amsthm, amscd, amsfonts, amssymb, graphicx, color, mathrsfs}
\usepackage[bookmarksnumbered, colorlinks, plainpages]{hyperref}
\usepackage[all]{xy}
\usepackage{slashed}

\newtheorem{innercustomthm}{Theorem}
\newenvironment{customthm}[1]
  {\renewcommand\theinnercustomthm{#1}\innercustomthm}
  {\endinnercustomthm}

\newtheorem{theorem}{Theorem}[section]
\newtheorem{lemma}[theorem]{Lemma}

\newtheorem{corollary}[theorem]{Corollary}
\theoremstyle{definition}

\newtheorem{example}[theorem]{Example}

\theoremstyle{remark}
\newtheorem{remark}[theorem]{Remark}
\numberwithin{equation}{section}

\begin{document}
	\setcounter{page}{1}
	
	\title[  Periodic and discrete multilinear  operators on Lebesgue spaces]{  $L^p$-boundedness and $L^p$-nuclearity of multilinear pseudo-differential operators on $\mathbb{Z}^n$ and the torus $\mathbb{T}^n$ }
	
	\author[D. Cardona]{Duv\'an Cardona}
	\address{
		Duv\'an Cardona:
		\endgraf
		Department of Mathematics  
		\endgraf
		Pontificia Universidad Javeriana
		\endgraf
		Bogot\'a-Colombia
		\endgraf
		Current address
		\endgraf
		Department of Mathematics: Analysis Logic and Discrete Mathematics
		\endgraf
		Ghent University
		\endgraf
		Ghent-Belgium
		\endgraf
		{\it E-mail address} {\rm duvanc306@gmail.com}
		}
	
	\author[V. Kumar]{Vishvesh Kumar}
	\address{
		Vishvesh Kumar:
		\endgraf
		Department of Mathematics             
		\endgraf
		Indian Institute of Technology Delhi
		\endgraf
		Huaz Khas, New Delhi- 110016
		\endgraf
		Delhi, India
		\endgraf
		Current address
		\endgraf
		Department of Mathematics: Analysis Logic and Discrete Mathematics
		\endgraf
		Ghent University
		\endgraf
		Ghent-Belgium
		\endgraf
		{\it E-mail address} {\rm vishveshmishra@gmail.com}
		\endgraf
	}


	
	
	\date{Received: xxxxxx;  Revised: yyyyyy; Accepted: zzzzzz.
	}
	
	\begin{abstract} In this article, we begin a systematic study of the boundedness and the nuclearity properties of multilinear periodic pseudo-differential operators and multilinear discrete pseudo-differential operators on $L^p$-spaces. First, we prove analogues of known multilinear Fourier multipliers theorems (proved by Coifman and Meyer, Grafakos, Tomita, Torres, Kenig, Stein, Fujita, Tao, etc.) in the context of periodic and discrete multilinear pseudo-differential operators. For this, we use the periodic analysis of pseudo-differential operators developed by Ruzhansky and Turunen. Later, we investigate the $s$-nuclearity, $0<s \leq 1,$ of periodic and discrete pseudo-differential operators. To accomplish this, we classify those $s$-nuclear multilinear integral operators on arbitrary Lebesgue spaces defined on $\sigma$-finite measures spaces. We also study similar properties for periodic Fourier integral operators. Finally, we present some applications of our study to deduce the periodic Kato-Ponce inequality and to examine the $s$-nuclearity of multilinear Bessel potentials as well as the $s$-nuclearity of periodic Fourier integral operators admitting suitable types of singularities.
	
	\end{abstract} \maketitle
	
	\tableofcontents
	
	\section{Introduction}
	In this investigation we study  the boundedness and the nuclearity of periodic and discrete multilinear pseudo-differential operators on $L^p$-spaces. These operators are defined as follows. If $m:\mathbb{T}^{n}\times \mathbb{Z}^{nr}\rightarrow \mathbb{C},$ $\mathbb{T}^n\cong [0,1)^n,$ is a measurable function, usually referred as symbol, then the periodic multilinear-pseudo-differential operator associated with symbol $m$ is the multilinear operator defined  by
	\begin{equation}
	T_m(f)(x)=\sum_{\xi\in \mathbb{Z}^{nr}}e^{i2\pi x\cdot (\xi_1+\xi_2+\cdots +\xi_r)  }m(x,\xi)(\mathscr{F}_{\mathbb{T}^n}{f}_{1})(\xi_1)\cdots (\mathscr{F}_{\mathbb{T}^n}{f}_{r})(\xi_r),\,x\in\mathbb{T}^n,
	\end{equation} where $\xi=(\xi_1,\xi_2,\cdots ,\xi_r),$ $f=(f_1,\cdots,f_r)\in \mathscr{D}(\mathbb{T}^n)^r,$ and $$ 
	(\mathscr{F}_{\mathbb{T}^n}{f}_{i})(\xi_i)=\int\limits_{\mathbb{T}^n}e^{-i2\pi x_i\xi_i}f_i(x_i)dx_i$$
	is the periodic Fourier transform of $f_i.$  On the other hand, if $a:\mathbb{Z}^{n}\times \mathbb{T}^{nr}\rightarrow \mathbb{C}$ is a measurable function, then the discrete multilinear-pseudo-differential operator associated with symbol $a$ is the multilinear operator defined  by
	\begin{equation}
	T_a(g)(\ell)=\int\limits_{ \mathbb{T}^{nr}}e^{i2\pi \ell\cdot   (\eta_1+\cdots +\eta_r) }a(\ell,\eta)(\mathscr{F}_{\mathbb{Z}^n}{g}_{1})(\eta_1)\cdots (\mathscr{F}_{\mathbb{Z}^n}{g}_{r})(\eta_r)d\eta,\,\,\ell\in\mathbb{Z}^n,
	\end{equation} where, $\eta=(\eta_1,\cdots,\eta_r),$ $g=(g_1,\cdots,g_r)\in \mathscr{S}(\mathbb{Z}^n)^r,$ and $$ 
	(\mathscr{F}_{\mathbb{Z}^n}{g}_{i})(\eta_i)=\sum\limits_{\ell_i\in\mathbb{Z}^n}e^{-i2\pi \ell_i\eta_i}g_i(\ell_i)$$
	is the discrete Fourier transform of $g_i.$ For $r\geq 2, $  these operators have been studied by V. $\textnormal{Catan}\breve{a}$ in \cite{Catana}. If $r=1,$ these quantization formulae can be reduced to the familiar expressions
	\begin{equation}\label{Agranovisch}
	T_m(f)(x)=\sum_{\xi\in \mathbb{Z}^{n}}e^{i2\pi x\cdot \xi  }\,m(x,\xi)\,(\mathscr{F}_{\mathbb{T}^n}{f})(\xi),\,x\in\mathbb{T}^n,
	\end{equation} and 
	\begin{equation}\label{Molahajloo}
	T_a(g)(\ell)=\int\limits_{ \mathbb{T}^{n}}e^{i2\pi \ell\cdot   \eta }\,a(\ell,\eta)(\mathscr{F}_{\mathbb{Z}^n}{g})(\eta)\,d\eta,\,\,\ell\in\mathbb{Z}^n.
	\end{equation} 
	Pseudo-differential operators represented by \eqref{Agranovisch} were defined by Volevich and Agranovich\cite{ag}. Later, this theory was developed by  McLean\cite{Mc}, Turunen and Vainikko\cite{tur},  and Ruzhansky and Turunen\cite{Ruz-2}. Nevertheless, the work of  Ruzhansky and Turunen \cite{Ruz-2,Ruz}, Cardona  \cite{Duvan3,Duvan4}, Delgado \cite{Profe} and Molahajloo and Wong \cite{s1,s2,m} provide some complementary results for this periodic theory. Ruzhansky and Turunen \cite{Ruz}, and Cardona, Messiouene and Senoussaoui \cite{CarMessiSeno} studied some of mapping properties for the periodic Fourier integral operators. 
	
	Pseudo-differential operators on $\mathbb{Z}^n$ (discrete pseudo-differential operators) were introduced by Molahajloo in \cite{m}, and some of their properties were developed in the last few years, see \cite{CardonaZn, DW,Rab1,Rab2,Rab3,Rab4,rod}. However, Botchway, Kibiti, and Ruzhansky in their recent fundamental work \cite{Ruzhansky} investigated the discrete pseudo-differential calculus and its applications to difference equations.
	
	In both cases (periodic pseudo-differential operators and discrete pseudo-differential operators), some of the results obtained in the above mentioned references are periodic or discrete analogues of well-known results for pseudo-differential operators on $\mathbb{R}^n,$ which are linear operators of the form
	\begin{equation}
	Af(x)=\int\limits_{\mathbb{R}^n}e^{i2\pi x\cdot \xi}\,a(x,\xi)\widehat{f}(\xi)\,d\xi,\,\,\,f\in \mathscr{D}(\mathbb{R}^n),
	\end{equation} where $\widehat{f}$ is the euclidean Fourier transform of $f,$ (see H\"ormander \cite{Hor2} for a complete treatment on this subject). The nuclearity of pseudo-differential operators on $\mathbb{R}^n$ has been treated in details by Aoki \cite{Aoki} and Rempala \cite{Rempala}. Multilinear pseudo-differential operators studied by several authors  including  B\'enyi,  Maldonado, Naibo, and  Torres, \cite{Benyi1,Benyi2}, Michalowski, Rule and Staubach, Miyachi and Tomita \cite{Michalowski,MiTo1,MiTo2,MiTo3} and references therein. It is worth mentioning that the multilinear analysis for multilinear multipliers of the form
	\begin{equation}
	T_a(f)(x)=\int\limits_{ \mathbb{R}^{nr}}e^{i2\pi x\cdot   (\eta_1+\cdots +\eta_r) }a(\eta)\widehat{f}_{1}(\eta_1)\cdots \widehat{f}_{r}(\eta_r)d\eta,\,\,x\in\mathbb{R}^n,
	\end{equation}
	born with the multilinear results by Coifman and Meyer (see \cite{CoifMeye1,CoifMeye2}), where it was shown that the condition
	\begin{equation}
	|\partial_{\eta_1}^{\alpha_1}\partial_{\eta_2}^{\alpha_2}\cdots\partial_{\eta_r}^{\alpha_r} a(\eta_1,\eta_2,\cdots,\eta_r)|\leq C_{\alpha}(|\eta_1|+|\eta_2|+\cdots +|\eta_r|)^{-|\alpha|}, 
	\end{equation}
	for sufficiently many multi-indices $\alpha=(\alpha_1,\alpha_2,\cdots,\alpha_r),$ implies the boundedness of $T_a$ from $L^{p_1}(\mathbb{R}^n)\times L^{p_2}(\mathbb{R}^n)\times \cdots \times L^{p_r}(\mathbb{R}^n) $ into  $ L^p(\mathbb{R}^n)$ provide that $1/p=1/p_1+
	1/p_2+\cdots +1/p_r,$ and $1\leq p_i,p<\infty.$   A generalization of this result was obtained by Tomita \cite{Tomitabilinear}, where it was proved that the multilinear  H\"ormander condition \begin{equation}
	\Vert a\Vert_{l.u.,H^s_{loc}(\mathbb{R}^{nr})}:=\sup_{k\in\mathbb{Z}}\Vert a(2^{k}\eta_1
	,2^k\eta_2,\cdots ,2^k\eta_r)\phi\Vert_{H^s}<\infty,\,\,\phi\in \mathscr{D}(0,\infty), \,\,s>\frac{nr}{2},
	\end{equation} implies the boundedness of $T_a$  from $L^{p_1}(\mathbb{R}^n)\times L^{p_2}(\mathbb{R}^n)\times \cdots \times L^{p_r}(\mathbb{R}^n)$ into  $ L^p(\mathbb{R}^n)$ provided that $1/p=1/p_1+
	1/p_2+\cdots +1/p_r,$ and $1\leq p_i,p<\infty.$ The case $r=1$ is known as the H\"ormander-Mihlin theorem. These multilinear theorems have been extended to Hardy spaces for some suitable  ranges  $0<p_i,p<\infty,$ in the works of Grafakos, Torres, Miyachi, Fujita, Tomita, Kenig, Stein, Muscalo, Thiele, Tao \cite{Fujita,Gra1,Gra2,Gra3,Gra4,Kenig,Tao}.  Periodic versions of these results will be presented in the next section as well as a study on the nuclearity of multilinear operators on $L^p$-spaces. Several applications will be considered including the toroidal Kato-Ponce inequality and some nuclear properties for multilinear Bessel potentials. 
	
	\section{Summary of Results} In this section, we state the main results of our investigation. First we consider those results on the boundedness of  periodic  and discrete multilinear operators which are periodic or discrete analogues of some  known results in the context of  multilinear operators on  $\mathbb{R}^n.$ In the last part of this section , we will describe those multilinear operators on $\mathbb{Z}^n$ and the torus $\mathbb{T}^n$ which admit a $s$-nuclear, $0<s\leq 1,$ extension on multilinear $L^p$-spaces.\\
	
	\noindent{\bf{Boundedness of periodic  multilinear  operators.}} Now, we present our main results related to the boundedness of periodic multilinear operators. At times, we denote $(x,\xi):=(x,\xi_1,\cdots,\xi_r)=x\cdot(\xi_1+\cdots +\xi_r).$
	
	\begin{customthm}{3.1} Let us assume that  $m$ satisfies the H\"ormander condition of order $s>0$
		\begin{equation*}
		\Vert m\Vert_{L^\infty(\mathbb{T}^n,l.u.,H^s_{loc}(\mathbb{R}^n))}:=\textnormal{ess}\sup_{x\in\mathbb{T}^n}\Vert m(x,\cdot)\Vert_{l.u.,\, H^s_{loc}}<\infty. 
		\end{equation*} Then the multilinear periodic pseudo-differential operator $T_m$ associated with $m$ extends to a bounded operator from $L^{p_1}(\mathbb{T}^n)\times L^{p_2}(\mathbb{T}^n)\times \cdots \times L^{p_r}(\mathbb{T}^n)$ into $L^p( \mathbb{T}^n)$ provided that $s>\frac{3nr}{2}$ and
		\begin{equation*}
		\frac{1}{p}=\frac{1}{p_1}+\cdots+ \frac{1}{p_r},\,\,\,1\leq p<\infty,\,1\leq p_i\leq \infty.
		\end{equation*}
	\end{customthm}
	With the help of previous result we prove the following fact. We use the notation    $$\langle \xi\rangle:=\max\{1,|\xi_1|+\cdots +|\xi_r|\},$$ for all  $\xi\in \mathbb{R}^{nr}.$

	\begin{customthm}{3.3}\label{Teoremapseuod}
		Let us assume that $m$ satisfies the discrete symbol inequalities
		\begin{equation*}
		\sup_{x\in\mathbb{T}^n}|\Delta_{\xi_1}^{\alpha_1}\Delta_{\xi_2}^{\alpha_2}\cdots \Delta_{\xi_r}^{\alpha_r}m(x,\xi_1,\cdots,\xi_r)|\leq C_\alpha\langle \xi\rangle^{-|\alpha|},
		\end{equation*} for all $|\alpha|:=  |\alpha_1|+\cdots +\cdots+|\alpha_r|\leq [3nr/2]+1.$ Then the periodic multilinear pseudo-differential operator $T_m$ extends to a bounded operator from $L^{p_1}(\mathbb{T}^n)\times L^{p_2}(\mathbb{T}^n)\times \cdots \times L^{p_r}(\mathbb{T}^n)$ into $L^p( \mathbb{T}^n)$ provided that 
		\begin{equation*}
		\frac{1}{p}=\frac{1}{p_1}+\cdots+ \frac{1}{p_r},\,\,\,1\leq p<\infty,\,1\leq p_i\leq \infty.
		\end{equation*}
	\end{customthm}
	Now,  we will consider Fourier integral operators with periodic phases.
	\begin{customthm}{3.5}
		Let $1< p<\infty$ and let $\phi$ be a real valued continuous function defined on $\mathbb{T}^n\times\mathbb{R}^{nr}.$ Let us assume that $a:\mathbb{T}^n\times\mathbb{R}^{nr}\rightarrow \mathbb{C}$ is a continuous bounded function and the multilinear Fourier integral operator
		\begin{equation*}
		Tf(x)=\int\limits_{\mathbb{R}^{nr}}e^{i\phi(x,\xi_1,\xi_2,\cdots, \xi_r)}a(x,\xi_1,\xi_2,\cdots, \xi_r)\widehat{f}_1(\xi_1)\cdots \widehat{f}_r(\xi_r)d\xi
		\end{equation*}  extends to a bounded multilinear operator from $L^{p_1}(\mathbb{R}^n)\times L^{p_2}(\mathbb{R}^n)\times\cdots\times L^{p_r}(\mathbb{R}^n)$ into $L^p(\mathbb{R}^n).$ Then the periodic multilinear Fourier integral operator
		\begin{equation*}\label{cp}
		Af(x):=\sum_{\xi\in\mathbb{Z}^{nr}}e^{i\phi(x,\xi_1,\xi_2,\cdots,\xi_r)}a(x,\xi_1,\xi_2,\cdots,\xi_r)(\mathscr{F}_{\mathbb{T}^n}{f}_{1})(\xi_1)\cdots (\mathscr{F}_{\mathbb{T}^n}{f}_{r})(\xi_r)
		\end{equation*} also extends to a  bounded multilinear operator from $L^{p_1}(\mathbb{T}^n)\times L^{p_2}(\mathbb{T}^n)\times\cdots \times L^{p_r}(\mathbb{T}^n)$ into $L^p(\mathbb{T}^n),$ provided that 
		\begin{equation*}
		\frac{1}{p_1}+\cdots+\frac{1}{p_r}=\frac{1}{p},\,\,1\leq p_i<\infty.
		\end{equation*}
		Moreover, there exists a positive constant  $C_p$ such that  $$\Vert A \Vert_{\mathscr{B}(L^{p_1}(\mathbb{T}^n)\times L^{p_2}(\mathbb{T}^n)\times\cdots\times L^{p_r}(\mathbb{T}^n),L^p(\mathbb{T}^n))}\leq C_p\Vert T\Vert_{\mathscr{B}(L^{p_1}(\mathbb{R}^n)\times L^{p_2}(\mathbb{R}^n)\times\cdots\times L^{p_r}(\mathbb{R}^n),L^p(\mathbb{R}^n))}.$$
	\end{customthm}
	We present the following multilinear  version of the Stein-Weiss  multiplier theorem (see Theorem 3.8 of Stein and Weiss \cite{SteWei}).

	\begin{customthm}{3.6}\label{vishvesh}
		Let $1< p<\infty$  and let  $a:\mathbb{R}^{nr}\rightarrow \mathbb{C}$ be a continuous bounded function. Let us assume that the multilinear Fourier multiplier operator
		\begin{equation*}
		Tf(x)=\int\limits_{\mathbb{R}^{nr}}e^{i2\pi(x,\xi_1,\xi_2,\cdots ,\xi_r)}a(\xi_1,\xi_2,\cdots ,\xi_r)\widehat{f}_1(\xi_1)\cdots \widehat{f}_r(\xi_r)d\xi
		\end{equation*} extends to a bounded multilinear operator from $L^{p_1}(\mathbb{R}^n)\times L^{p_2}(\mathbb{R}^n)\times\cdots \times L^{p_r}(\mathbb{R}^n)$ into $L^p(\mathbb{R}^n).$ Then the periodic multilinear Fourier multiplier 
		\begin{equation*}\label{cp}
		Af(x):=\sum_{\xi\in\mathbb{Z}^{nr}}e^{i2\pi(x, \xi_1,\xi_2,\cdots,\xi_r)}a(\xi_1,\xi_2,\cdots,\xi_r)(\mathscr{F}_{\mathbb{T}^n}{f}_{1})(\xi_1)\cdots (\mathscr{F}_{\mathbb{T}^n}{f}_{r})(\xi_r)
		\end{equation*} also extends to a  bounded multilinear operator from $L^{p_1}(\mathbb{T}^n)\times L^{p_2}(\mathbb{T}^n)\times\cdots\times L^{p_r}(\mathbb{T}^n)$ into $L^p(\mathbb{T}^n),$ provided that 
		\begin{equation*}
		\frac{1}{p_1}+\cdots+\frac{1}{p_r}=\frac{1}{p},\,\,1\leq p_i<\infty.
		\end{equation*}
		Moreover, there exists a positive constant $C_p$ such that the following  inequality holds. $$\Vert A \Vert_{\mathscr{B}(L^{p_1}(\mathbb{T}^n)\times L^{p_2}(\mathbb{T}^n)\times\cdots \times L^{p_r}(\mathbb{T}^n),L^p(\mathbb{T}^n))}\leq C_p\Vert T\Vert_{\mathscr{B}(L^{p_1}(\mathbb{R}^n)\times L^{p_2}(\mathbb{R}^n)\times \cdots \times L^{p_r}(\mathbb{R}^n),L^p(\mathbb{R}^n))}.$$
	\end{customthm}
	
	The multilinear Coifman-Meyer theorem together with Theorem \ref{vishvesh} allow us to prove the following multilinear estimate.
	\begin{customthm}{3.7} Let $T_m$ be a periodic multilinear Fourier multiplier. Let us assume that the symbol $m$ satisfies the estimates
		
		$$|\Delta_{\xi_1}^{\alpha_{1} }\cdots \Delta_{\xi_r}^{\alpha_{r}  }m(\xi_1,\xi_2,\cdots,\xi_r)|\leq C_\alpha \langle \xi \rangle^{-|\alpha_1|-\cdots -|\alpha_r|},\,\,\,|\alpha|\leq [\frac{nr}{2}]+1.$$
		Then the operator $T_m$ extends to a  bounded multilinear operator from $L^{p_1}(\mathbb{T}^n)\times L^{p_2}(\mathbb{T}^n)\times\cdots \times L^{p_r}(\mathbb{T}^n)$ into $L^p(\mathbb{T}^n),$ provided that 
		\begin{equation*}
		\frac{1}{p_1}+\cdots+\frac{1}{p_r}=\frac{1}{p},\,\,1\leq p_i<\infty.
		\end{equation*}
		
	\end{customthm}
	For $r=1,$ Theorem \ref{Teoremapseuod} was generalized  by Delgado in \cite{Profe} to more general periodic H\"ormander classes of limited regularity. The condition on the number of discrete derivatives in the preceding result can be relaxed if we assume regularity in $x.$  We show it in the following theorem. 
	\begin{customthm}{3.8}\label{limitedregularity} Let $T_m$ be a periodic multilinear pseudo-differential operator. Let us assume that $m$ satisfies toroidal conditions of the type,
		\begin{equation*}
		|\partial_{x}^\beta\Delta_{\xi_1}^{\alpha_{1} }\cdots \Delta_{\xi_r}^{\alpha_{r}  } m(x,\xi_1,\xi_2,\cdots,\xi_r)|\leq C_\alpha \langle \xi \rangle^{-|\alpha_1|-\cdots -|\alpha_r|},
		\end{equation*} where $|\alpha|\leq [\frac{nr}{2}]+1,$ and $|\beta|\leq [\frac{n}{p}]+1.$ Then $T_m$ extends to a  bounded multilinear operator from $L^{p_1}(\mathbb{T}^n)\times L^{p_2}(\mathbb{T}^n)\times\cdots \times L^{p_r}(\mathbb{T}^n)$ into $L^p(\mathbb{T}^n),$ provided that 
		\begin{equation*}
		\frac{1}{p_1}+\cdots+\frac{1}{p_r}=\frac{1}{p},\,\,1\leq p_i<\infty.
		\end{equation*}
	\end{customthm}

	 Let us observe that Theorem \ref{limitedregularity} has been proved for $r=1$ in Ruzhansky and Turunen \cite{Ruz}. Now we discuss some applications of our multilinear analysis. 
	Theorem \ref{Teoremapseuod} applied to the bilinear operator
	\begin{equation}
	    B_s(f,g):=J^s(f\cdot g),
	\end{equation}
	where $J^s$ is the periodic  fractional derivative operator $(\mathcal{L})^{s/2},$ or the periodic  Bessel potential of order $s>0,$
	$(1+\mathcal{L})^{s/2},$ implies  the (well known) periodic Kato-Ponce inequality (see Muscalu and Schlag \cite{Muscaluperiodic}):
	
\begin{equation}\label{periodickatoponce}
    \Vert J^s( f \cdot g)\Vert_{L^r(\mathbb{T}^n)}\lesssim \Vert J^s f\Vert_{L^{p_1}(\mathbb{T}^n)}\Vert g\Vert_{L^{q_1}(\mathbb{T}^n)}+\Vert  f\Vert_{L^{p_2}(\mathbb{T}^n)}\Vert J^s g\Vert_{L^{q_2}(\mathbb{T}^n)}
\end{equation}	
	where $\frac{1}{p_1}+\frac{1}{q_1}=\frac{1}{p_2}+\frac{1}{q_2}=\frac{1}{r},$ $1\leq p_i,q_i\leq\infty,$ $1<r<\infty,$ and  $\mathcal{L}=-\frac{1}{4\pi^{2}}(\sum_{j=1}^n\partial_{\theta_j}^2)$ is the Laplacian on the torus.
We develop this application of Theorem \ref{Teoremapseuod} to PDEs in Remark \ref{limitedregularity'}. \\
\\

	\noindent{\bf{Boundedness of discrete  multilinear pseudo-differential operators.}}
	Our main result about the boundedness of discrete multilinear pseudo-differential operators is the following.
	
	\begin{customthm}{3.10}
		Let  $\sigma\in L^\infty(\mathbb{Z}^{n}, C^{2\varkappa} (\mathbb{T}^{nr})).$ Let us assume that $\sigma$ satisfies the following discrete inequalities
		\begin{equation*}
		|\partial_{\xi}^{\beta} \sigma(\ell,\xi)|\leq C_{\beta},\,\,\ell\in \mathbb{Z}^{n},\,\,\xi\in\mathbb{T}^{nr},\,\sigma(\ell,\xi)=\sigma(\ell) (\xi) ,
		\end{equation*} for all $\beta$ with $|\beta|= 2\varkappa.$ Then $T_\sigma$ extends to a bounded operator from $L^{p_1}(\mathbb{Z}^n)\times L^{p_2}(\mathbb{Z}^n)\times \cdots\times L^{p_r}(\mathbb{Z}^n )$ into $L^{s}(\mathbb{Z}^n )$ provided that $1\leq p_j\leq p \leq\infty,$ and $$ \frac{1}{s}-\frac{1}{p}<\frac{2\varkappa}{nr}-1.$$ 
		\end{customthm} 
	As a consequence of the previous result, 	if $\sigma\in C^{2\varkappa}(\mathbb{Z}^{n}\times \mathbb{T}^{n}),$  satisfies the following discrete inequalities
		\begin{equation*} \label{derivedcor++}
		|\partial_{\xi}^{\beta} \sigma(\ell,\xi)|\leq C_{\beta},\,\,\ell\in \mathbb{Z}^{n},\,\,\xi\in\mathbb{T}^{n},  
		\end{equation*} for all $\beta$ with $|\beta|= 2\varkappa,$ then $T_\sigma$ extends to a bounded operator from $L^{p}(\mathbb{Z}^n)$ into $L^{p}(\mathbb{Z}^n )$ provided that $1\leq p\leq \infty,$ and $\varkappa>n/2.$ This  implies the $L^p$-boundedness of pseudo-differential operators associated to the discrete H\"ormander class $S^0_{0,0}(\mathbb{Z}^n\times \mathbb{T}^n )$ introduced in Botchway, Kibiti and Ruzhansky \cite{Ruzhansky}.\\

	\noindent{\bf{$s$-Nuclearity for periodic and discrete multilinear pseudo-differential operators.}}
	In order to study those multilinear operators admitting $s$-nuclear extensions, we prove the following multilinear version of a result by Delgado, on the nuclearity of integral operators on Lebesgue spaces (see \cite{Delgado,D2}). So, in the following multilinear theorem we characterize those  $s$-nuclear (multilinear) integral operators on arbitrary ($\sigma$-finite) measure spaces $(X,\mu)$. 
	
	\begin{customthm}{4.4}
		Let $(X_i, \mu_i), 1 \leq i \leq r$ and $(Y, \nu)$ be $\sigma$-finite measure spaces. Let $1 \leq p_i,p <\infty, 1 \leq i \leq r$ and let $p_i', q$ be such that $\frac{1}{p_i}+ \frac{1}{p_i'}=1, \frac{1}{p}+\frac{1}{q}=1$ for $1 \leq i \leq r.$ Let $T: L^{p_1}(\mu) \times L^{p_2}(\mu_2) \times \cdots \times L^{p_r}(\mu_r) \rightarrow L^p(\nu)$ be a multilinear operator. Then $T$ is a $s$-nuclear, $0<s \leq 1,$ operator if and only if there exist sequences $\{g_n\}_n$ with $g_n=(g_{n1}, g_{n2}, \ldots, g_{nr})$ and $\{h_n\}_n$ in $L^{p_1'}(\mu_1) \times L^{p_2'}(\mu_2) \times \cdots \times L^{p_r'}(\mu_r)$ and $L^p(\nu)$ respectively, such that $\sum_{n} \|g_n\|_{L^{p_1'}(\mu_1) \times L^{p_2'}(\mu_2) \times \cdots \times L^{p_r'}(\mu_r)}^s \|h_n\|_{L^p(\nu)}^s<\infty$ and for all $f = (f_1,f_2,\ldots, f_r) \in L^{p_1}(\mu) \times L^{p_2}(\mu_2) \times \cdots \times L^{p_r}(\mu_r) $ we have
		\begin{align*}
		(Tf)(y)&= \int_{X_1,X_2,\cdots,X_r} \left( \sum_{n=1}^\infty g_n(x)\, h_n(y) \right) f(x)\, d(\mu_1 \otimes \mu_2 \otimes \cdots \otimes \mu_r)(x) \\
		& = \int_{X_1} \int_{X_2} \cdots \int_{X_r} \left( \sum_{n=1}^\infty g_{n1}(x_1)\, g_{n2}(x_2)\ldots g_{nr}(x_r)\, h_n(y) \right) \\
		& \hspace{4cm} \times f_1(x_1) \,f_2(x_2) \ldots f_r(x_r)\, d\mu_1(x_1)\, d\mu_2(x_2) \cdots d\mu_r(x_r)
		\end{align*}
		for almost every $y \in Y.$
	\end{customthm}
	
	This criterion applied to discrete and periodic operators gives the following characterizations. 
	
	\begin{customthm}{5.1} 
		Let $a$ be a measurable function defined on $\mathbb{Z}^n \times \mathbb{T}^{nr}.$ The multilinear pseudo-differential operator $T_a: L^{p_1}(\mathbb{Z}^n)\times L^{p_2}(\mathbb{Z}^n) \times \cdots L^{p_r}(\mathbb{Z}^n) \rightarrow L^{p}(\mathbb{Z}^n), 1 \leq p_i< \infty,$ for all $1 \leq i \leq r$ is a $s$-nuclear, $0 < s \leq 1, $ operator if and only if  the following decomposition holds:
		$$a(x,\xi)=e^{-i2\pi \tilde{x}\cdot \xi}\sum_{k}   h_k(x)\mathscr{F}_{\mathbb{Z}^{nr}}(g_k)(-\xi),\xi\in\mathbb{T}^{nr},x\in\mathbb{Z}^{n},$$ 
		where $\tilde{x}= (x,x, \ldots,x)  \in (\mathbb{Z}^n)^r;$ $\{h_k\}_k$ and $\{g_k\}_k$ with $g_k= (g_{k1}, g_{k2}, \ldots, g_{kr})$  are two sequences in $L^{p}(\mathbb{Z}^n)$ and $L^{p_1'}(\mathbb{Z}^n)\times L^{p_2'}(\mathbb{Z}^n) \times \cdots \times L^{p_r'}(\mathbb{Z}^n)$ respectively such that $\sum_{n=1}^\infty \|h_n\|_{L^{p}(\mathbb{Z}^n)}^s \|g_n\|_{L^{p_1'}(\mathbb{Z}^n)\times L^{p_2'}(\mathbb{Z}^n) \times \cdots \times L^{p_r'}(\mathbb{Z}^n)}^s <\infty.$
	\end{customthm}
	
	\begin{customthm}{5.4} Let $m$ be a measurable function on $\mathbb{T}^n \times \mathbb{Z}^{nr}.$ Then the mutlilinear pseudo-differential operator $T_m:L^{p_1}(\mathbb{T}^n) \times \cdots \times L^{p_r}(\mathbb{T}^n) \rightarrow L^p(\mathbb{T}^n),$ $1 \leq p_i,p < \infty $ for $1\leq i \leq r,$ is a $s$-nuclear, $0 < s \leq 1, $ operator if and only if there exist two sequences $\{g_k\}_k$ with $g_{k}= \left(g_{k1},g_{k2}, \ldots, g_{kr}\right)$ and $\{h_k\}_k$ in $L^{p_1'}(\mathbb{T}^n) \times \cdots \times L^{p_r'}(\mathbb{T}^n),\, \frac{1}{p_i}+\frac{1}{p_i'}=1$ for $1\leq i \leq r$ and $L^p(\mathbb{T}^n)$  respectively such that  $\sum_{k} \|g_k\|_{L^{p_1'}(\mathbb{T}^n) \times \cdots \times L^{p_r'}(\mathbb{T}^n)}^s \|h_k\|_{L^p(\mathbb{T})}^s <\infty $ and $$ m(x, \eta)= e^{-i2 \pi \tilde{x} \cdot \eta} \sum_{k} h_k(x)\, (\mathscr{F}_{\mathbb{T}^{nr}}g_k)(-\eta),\,\, \eta \in \mathbb{Z}^{nr}$$ where $\tilde{x}= (x,x,\ldots,x) \in \mathbb{T}^{nr}.$
	\end{customthm}
	Now, we present the following sharp result on the $s$-nuclearity of periodic Fourier integral operators.
	
	\begin{customthm}{5.5}  Let us consider the real-valued function $\phi:\mathbb{T}^n\times \mathbb{Z}^{nr}\rightarrow \mathbb{R}.$
		Let us consider the Fourier integral operator \begin{equation*}\label{cp2}
		Af(x):=\sum_{\xi\in\mathbb{Z}^{nr}}e^{i\phi(x,\xi_1,\xi_2,\cdots,\xi_r)}a(x,\xi_1,\xi_2,\cdots,\xi_r)(\mathscr{F}_{\mathbb{T}^n}{f}_{1})(\xi_1)\cdots (\mathscr{F}_{\mathbb{T}^n}{f}_{r})(\xi_r) 
		\end{equation*} with symbol satisfying the summability condition
		\begin{equation*}
		\sum_{\xi \in  \mathbb{Z}^{nr} }\Vert a(\cdot,\xi_1,\xi_2,\cdots,\xi_r)\Vert^s_{L^p(\mathbb{T}^n)}<\infty.
		\end{equation*} Then $A$ extends to a $s$-nuclear, $0<s \leq 1,$ operators from $L^{p_1}(\mathbb{T}^n) \times \cdots \times L^{p_r}(\mathbb{T}^n)$ into $L^p(\mathbb{T}^n)$ provided that $1\leq p_j<\infty,$ and $1\leq p\leq \infty.$
	\end{customthm}
	The previous theorem will be applied to analyze multilinear Bessel potentials \begin{equation*}
	(I+ \mathscr{L})^{-\frac{\alpha}{2}}:=((I+\mathcal{L})^{ -\frac{\alpha_1}{2}},\cdots,(1+\mathcal{L})^{ -\frac{\alpha_r}{2}  }),
	\end{equation*}  and Fourier integral operators with singular symbols of the form
	\begin{equation*}
	a(x,\xi):=\frac{1}{|x|^\rho}\kappa(\xi),\,\,x\in \mathbb{T}^n,\,x\neq 0,\,\xi\in\mathbb{Z}^{nr},\,\,0<\rho<n/p,\,1<p<\infty .
	\end{equation*}
	The results of this article were announced in the note \cite{CardonaKumar}.
	\section{Periodic and discrete multilinear pseudo-differential operators on Lebesgue spaces}
	
	\subsection{H\"ormander condition for pseudo-differential operators on periodic Lebesgue spaces}
	In this section, we provide some results concerning to the boundedness of the multilinear periodic pseudo-differential operators on Lebesgue spaces. Note that symbols satisfying the H\"ormander condition lie in locally uniformly Sobolev spaces. If $\Vert a\Vert_{H^s(\mathbb{R}^n)}=\Vert \langle\xi\rangle^{s}|\widehat{a}(\cdot)| \Vert_{L^2(\mathbb{R}^n)},$ denotes the Sobolev norm of a function $a,$ a symbol $\sigma$ belongs to the locally uniformly Sobolev space $H^s,$ if
	\begin{equation}
	\Vert \sigma\Vert_{l.u.,H^s(\mathbb{R}^n)}=\sup_{k\in\mathbb{Z}}\Vert \sigma(2^{k}\cdot)\phi\Vert_{H^s(\mathbb{R}^n)}=\sup_{k\in\mathbb{Z}}2^{k(s-\frac{n}{2})}\Vert \sigma(\cdot)\phi(2^{-k}\cdot)\Vert_{H^s(\mathbb{R}^n)}<\infty,
	\end{equation}
	where $\phi$ is a function with compact support on $\mathbb{R}^n.$  The space defined by the previous norm is independent of the choice of the function $\phi.$  So, we will consider a test function $\varphi$ supported in $\{1/2<|\xi|<4\}$ with $\varphi=1$ on $\{1\leq |\xi|\leq 2\}.$ Throughout this article, we write $\phi(\xi)=\varphi(|\xi|).$ We freely use the notations and concepts introduced above and in the previous section. We would like to begin this section with the following result.
	
	\begin{theorem}\label{HorCondition} Let us assume that $m$ satisfies the H\"ormander condition of order $s>0$
		\begin{equation}
		\Vert m\Vert_{L^\infty(\mathbb{T}^n,l.u.,H^s_{loc}(\mathbb{R}^n))}:=\textnormal{ess}\sup_{x\in\mathbb{T}^n}\Vert m(x,\cdot)\Vert_{l.u., H^s_{loc}}<\infty. 
		\end{equation} Then the multilinear periodic pseudo-differential operator $T_m$ associated with $m$ extends to a bounded operator from $L^{p_1}(\mathbb{T}^n)\times L^{p_2}(\mathbb{T}^n)\times \cdots \times L^{p_r}(\mathbb{T}^n)$ into $L^p( \mathbb{T}^n)$ provided that $s>\frac{3nr}{2}$ and
		\begin{equation}
		\frac{1}{p}=\frac{1}{p_1}+\cdots+ \frac{1}{p_r} ,\,\,\,1\leq p<\infty,\,1\leq p_i\leq \infty.
		\end{equation}
	\end{theorem}
	\begin{proof} The proof consists of two steps. First, we will prove that   $T_m$ extends to a bounded operator from  $L^{p}\times L^{\infty}\times\cdots \times L^{\infty}\times L^{\infty}$ into $L^{p}(\mathbb{T}^n),$ for all $1\leq p<\infty.$ Secondly, we conclude the proof by the real multilinear interpolation.
		Let us assume that $T_m$ satisfies the multilinear H\"ormander condition for $s>\frac{3nr}{2}.$ In order to prove the boundedness of $T_m$ from $L^{p}\times L^{\infty}\times\cdots \times L^{\infty}\times L^{\infty}$ into $L^{p}(\mathbb{T}^n),$ we will use the following property, 
	$
		\Vert T_m f\Vert_{L^{p}(\mathbb{T}^n)}=\sup_{\Vert g\Vert_{L^{p'}}=1}|\langle T_m f,\overline{g}\rangle|.
		$
		Now, for $f:=(f_{1},f_{2},\cdots ,f_r)\in\mathscr{D}(\mathbb{T}^n)^{r}$ we have
		\begin{equation}
		|\langle T_m f,\overline{g} \rangle|\leq |\langle T_0 f,\overline{g} \rangle |+\sum_{k=0}^{\infty}|\langle T_{m(k)} f,\overline{g} \rangle|,
		\end{equation}
		where $T_{m(k)}$ is the multilinear periodic pseudo-differential operator associated with symbol $$m_{k}(x,\xi)=m(x,\xi)\cdot 1_{[2^{k},2^{k+1})}(|\xi|)=m(x,\xi)\cdot \varphi_k(|\xi|)=m(x,\xi)\cdot \phi_k(\xi),$$ and $T_{0}$ is the operator associated  with symbol $m(x,0)\delta_{\xi,0}.$ For $z_j\in \mathbb{T}^n,$ $z=(z_1,z_2,\cdots ,z_r)\in \mathbb{T}^{nr},$  the inversion formula for the Fourier transform gives
		\begin{align*}
		&|\langle T_{m(k)} f,\overline{g} \rangle |  =\left| \int_{\mathbb{T}^n}    T_{m(k)}f(x)g(x)dx\right| \\
		& =| \int_{\mathbb{T}^n} \sum_{ 2^{k}\leq |\xi|<2^{k+1} }m_k(x,\xi)\widehat{f}_1({\xi_1})\cdots \widehat{f}_r({\xi_r}) e^{i2\pi x(\xi_1+\cdots+ \xi_r)}g(x)dx|\\
		&\leq \sum_{ 2^{k}\leq |\xi|<2^{k+1} } | \int_{\mathbb{T}^n} m_k(x,\xi)\widehat{f}_1({\xi_1})\cdots \widehat{f}_r({\xi_r}) e^{i2\pi x(\xi_1+\cdots +\xi_r)} g(x)dx|\\
		&=  \sum_{ 2^{k}\leq |\xi|<2^{k+1} } | \int_{\mathbb{T}^n} \int_{\mathbb{R}^{n r}}e^{-i2\pi z\xi}\mathscr{F}^{-1}[m_k(x,\cdot)](z)dz\\
		&\hspace{4cm}\times\widehat{f}_1({\xi_1})\cdots \widehat{f}_r({\xi_r}) e^{i2\pi x(\xi_1+\cdots +\xi_r)}g(x)dx|\\
		&\leq  \sum_{ 2^{k}\leq |\xi|<2^{k+1} } \sup_{x\in\mathbb{T}^n} \int_{\mathbb{R}^{nr}}|\mathscr{F}^{-1}[m_k(x,\cdot)]|(z)dz \Vert f_{1}\Vert_{L^{p}} \prod_{j=2}^{r} \Vert f_{j}\Vert_{L^{\infty}}\cdot \Vert g\Vert_{L^{p'}}\\
		\end{align*}
		Consequently,
		\begin{align*}
		&|\langle T_{m(k)} f,\overline{g} \rangle|  \leq \sum_{ 2^{k}\leq |\xi|<2^{k+1} } \sup_{x\in\mathbb{T}^n}  \left( \int_{\mathbb{R}^{nr}}\langle z\rangle^{2s}|\mathscr{F}^{-1}[m_k(x,\cdot)](z)|^2dz\right)^\frac{1}{2}\left(\int_{\mathbb{R}^{nr}}\langle z\rangle^{-2s} dz\right)^\frac{1}{2}\\
		&\hspace{2cm}\times \Vert f_{1}\Vert_{L^{p}} \prod_{j=2}^{r} \Vert f_{j}\Vert_{L^{\infty}}\cdot \Vert g\Vert_{L^{p'}},\\
		\end{align*}
		where we have used that  $\left(\int_{\mathbb{R}^{nr}}\langle z\rangle^{-2s} dz\right)^\frac{1}{2}<\infty$ for $s>\frac{nr}{2}.$  Hence we have the following estimate for the norm of $T_{m(k)}$,
		\begin{align*}
		&\Vert T_{m(k)} \Vert_{\mathscr{B}(L^p\times (L^{\infty})^{r-1},\,L^p)} \lesssim  \sum_{ 2^{k}\leq |\xi|<2^{k+1} } \sup_{x\in\mathbb{T}^n}  \left( \int_{\mathbb{R}^{nr}}\langle z\rangle^{2s}|\mathscr{F}^{-1}[m_k(x,\cdot)](z)|^2dz\right)^\frac{1}{2}\\
		& \leq \sum_{ 2^{k}\leq |\xi|<2^{k+1} }2^{-k(s-\frac{nr}{2})}\Vert m\Vert_{l.u.,H^s_{loc}} \asymp  2^{knr}2^{-k(s-\frac{nr}{2})}\Vert m\Vert_{l.u.,H^s_{loc}}\\
		& =  2^{-k(s-\frac{3nr}{2})}\Vert m\Vert_{l.u.,H^s_{loc}}.
		\end{align*} 
		So, we obtain the following upper bound for the series
		\begin{align*}
		\sum_{k=1}^{\infty} \Vert T_{m(k)}  \Vert_{\mathscr{B}(L^p\times (L^{\infty})^{r-1},\,L^p)}\lesssim  \Vert m \Vert_{l.u.,{H}^s_{loc}}\times \sum_{k=1}^{\infty}2^{-k(s-\frac{3nr}{2})}
		\end{align*}
		which converges provided that $s>\frac{3nr}{2}.$ Now, it is easy to see that $$\Vert T_{0} \Vert_{\mathscr{B}(L^p\times (L^{\infty})^{r-1},\,L^p)}\lesssim \Vert m(\cdot,0) \Vert_{L^{\infty}(\mathbb{T}^{n})}.$$
		As a consequence we get $$  \Vert T_m \Vert_{\mathscr{B}(L^p\times (L^{\infty})^{r-1},\,L^p)}\leq C(\Vert m \Vert_{l.u.,{H}^s_{loc}}+\Vert m(\cdot,0) \Vert_{L^{\infty}(\mathbb{T}^{n})} ).$$
		So, we conclude the proof of first step. Now, a similar argument shows that the H\"ormander condition of order $s>\frac{3nr}{2}$ implies the boundedness of $T_m$ from $T_m$ from $L^{\infty}\times\cdots \times L^{p}\times\cdots \times L^{\infty}$ into $L^{p}(\mathbb{T}^n),$ where the $L^p$-space appears only in one of the $k$-positions of the product, for $k=2,3,\cdots,r.$ Now, the boundedness of $T_m$ from $L^{p_1}\times L^{p_2}\times \cdots \times L^{p_r}$ into $L^p( \mathbb{T}^n)$ provided that
		\begin{equation}
		\frac{1}{p}=\frac{1}{p_1}+\cdots+ \frac{1}{p_r}
		\end{equation} follows by using the real multilinear interpolation.
	\end{proof}
	
	\subsection{$L^p$-boundedness for periodic multilinear pseudo-differential operators in periodic Kohn-Nirenberg classes}
	
	In this subsection we study the $L^p$-boundedness for periodic multilinear pseudo-differential operators. The multilinear symbols   are considered in periodic Kohn-Nirenberg classes. More precisely we consider symbols satisfying inequalities of the type,
	\begin{equation}
	\sup_{x\in\mathbb{T}^n}|\Delta_{\xi_1}^{\alpha_1}\Delta_{\xi_2}^{\alpha_2}\cdots \Delta_{\xi_r}^{\alpha_r}m(x,\xi_1,\cdots,\xi_r)|\leq C_\alpha\langle \xi\rangle^{m-|\alpha|},
	\end{equation} for all $|\alpha|:=|\alpha_1|+\cdots+|\alpha_r| \leq N_1,$ $1\leq N_1<\infty.$ Here $\Delta_{\xi_j}^{\alpha_j}:= \Delta_{\xi_{j1}}^{\alpha_{j1}}\cdots  \Delta_{\xi_{jn}}^{\alpha_{jn}} $ is the composition of  difference operators $\Delta_{\xi_{jk}}^{\alpha_{jk}}$, $1\leq k\leq n,$ defined by
	$$  \Delta_{\xi_{jk}}f(x)=f(x+e_k)-f(x),\,\,x\in\mathbb{Z}^n,\,\,e_{k}(j):=\delta_{kj}.  $$
	Our starting point is the following result due to Ruzhansky and Turunen (see Corollary 4.5.7 of \cite{Ruz}).
	
	\begin{lemma}\label{IClases}\ Let $0\leq \delta \leq 1,$ $0\leq \rho<1.$ Let $a:\mathbb{T}^n\times \mathbb{R}^n\rightarrow \mathbb{C}$ satisfying
		\begin{equation}
		\lvert  \partial_\xi ^\alpha \partial_x^\beta a (x, \xi)\rvert \leq C^{(1)}_{a\alpha \beta m}\langle \xi \rangle^{m- \rho \lvert \alpha \rvert + \delta \lvert \beta \rvert} ,
		\end{equation}
		for $|\alpha|\leq N_{1}$ and $|\beta|\leq N_{2}.$ Then the restriction $\tilde{a}=a|_{\mathbb{T}^n\times \mathbb{Z}^n}$ satisfies the estimate
		\begin{equation}
		\lvert  \Delta_\xi ^\alpha \partial_x^\beta \tilde{a} (x, \xi)\rvert \leq C_{a\alpha \beta m}C^{(1)}_{a\alpha \beta m}\langle \xi \rangle^{m- \rho \lvert \alpha \rvert + \delta \lvert \beta \rvert} ,
		\end{equation}
		for $|\alpha|\leq N_{1}$ and $|\beta|\leq N_{2}.$  The converse holds true, i.e, if a symbol $\tilde{a}(x,\xi)$ on $\mathbb{T}^n\times \mathbb{Z}^n$ satisfies  $(\rho,\delta)$-inequalities of the form 
		\begin{equation}
		\lvert  \Delta_\xi ^\alpha \partial_x^\beta \tilde{a} (x, \xi)\rvert \leq C^{(2)}_{a\alpha \beta m}\langle \xi \rangle^{m- \rho \lvert \alpha \rvert + \delta \lvert \beta \rvert} ,
		\end{equation}
		then $\tilde{a}(x,\xi)$ is  the restriction of a symbol $a(x,\xi)$ on $\mathbb{T}^n\times \mathbb{R}^n$ satisfying estimates of the type
		\begin{equation}
		\lvert  \partial_\xi ^\alpha \partial_x^\beta a (x, \xi)\rvert \leq C_{a\alpha \beta m}C^{(2)}_{a\alpha \beta m}\langle \xi \rangle^{m- \rho \lvert \alpha \rvert + \delta \lvert \beta \rvert} .
		\end{equation}
	\end{lemma}

	Now, we present the main result of this subsection.

	\begin{theorem}\label{Teoremapseuod'}
		Let us assume that $m$ satisfies the discrete symbol inequalities
		\begin{equation}
		\sup_{x\in\mathbb{T}^n}|\Delta_{\xi_1}^{\alpha_1}\Delta_{\xi_2}^{\alpha_2}\cdots \Delta_{\xi_r}^{\alpha_r}m(x,\xi_1,\cdots,\xi_r)|\leq C_\alpha\langle \xi\rangle^{-|\alpha|},
		\end{equation} for all $|\alpha|:=  |\alpha_1|+\cdots +\cdots+|\alpha_r|\leq [3nr/2]+1.$ Then the periodic multilinear pseudo-differential operator $T_m$ extends to a bounded operator from $L^{p_1}(\mathbb{T}^n)\times L^{p_2}(\mathbb{T}^n)\times \cdots \times L^{p_r}(\mathbb{T}^n)$ into $L^p( \mathbb{T}^n)$ provided that 
		\begin{equation}
		\frac{1}{p}=\frac{1}{p_1}+\cdots+ \frac{1}{p_r}, \,\,\,1\leq p<\infty,\,1\leq p_i\leq \infty.
		\end{equation}
	\end{theorem}
	\begin{proof}
		Let us assume that $m$ satisfies \begin{equation}
		\sup_{x\in\mathbb{T}^n}|\Delta_{\xi_1}^{\alpha_1}\Delta_{\xi_2}^{\alpha_2}\cdots \Delta_{\xi_r}^{\alpha_r}m(x,\xi_1,\cdots,\xi_r)|\leq C_\alpha\langle \xi\rangle^{-|\alpha_1|-\cdots -|\alpha_r|},
		\end{equation} for all $|\alpha|:=  |\alpha_1|+\cdots +|\alpha_r|\leq [3nr/2]+1.$ Define symbol $a_x(\xi):=m(x,\xi),$ for every $x\in \mathbb{T}^n,$ $\xi:=(\xi_1,\cdots ,\xi_r)\in \mathbb{Z}^{nr}.$ Clearly,  we have 
		\begin{equation}
		\sup_{x\in\mathbb{T}^n}|\Delta_{\xi}^{\alpha}a_{x}(\xi)|\leq C_\alpha\langle \xi\rangle^{-|\alpha|},
		\end{equation} for all $|\alpha|\leq [3nr/2]+1.$
		Now, by using Lemma \ref{IClases}, there exists $\tilde{a}_{x}$ such that $a_x=\tilde{a}_{x}|_{  {\mathbb{Z}}^{nr}}$ and 
		\begin{equation}
		|\partial^\alpha_\xi \tilde{a}_x(\xi)|\leq C_{\alpha'}\cdot C_\alpha\langle \xi\rangle^{-|\alpha|},\,\,\,|\alpha|\leq [3nr/2]+1.
		\end{equation}
		Now, taking into account that
		\begin{equation}
		\Vert \tilde{a}_x\Vert_{l.u.,H^s_{loc}(\mathbb{R}^{nr})}\lesssim \sup_{\xi\in \mathbb{R}^n}\sup_{|\alpha|\leq [3nr/2]+1}\langle \xi\rangle^{|\alpha|}|\partial^\alpha_\xi \tilde{a}_x(\xi)|\leq \sup_{|\alpha|\leq [3nr/2]+1} C_{\alpha'}\cdot C_\alpha,
		\end{equation} for $\frac{3nr}{2}<s<[3nr/2]+1,$ we deduce
		\begin{equation}
		\Vert \tilde{a}(x,\xi)\Vert_{L^\infty(\mathbb{T}^n_x,l.u.,H^s_{loc}(\mathbb{R}^n_\xi))}:=\textnormal{ess}\sup_{x\in\mathbb{T}^n}\Vert \tilde{a}_x\Vert_{l.u.,H^s_{loc}(\mathbb{R}^{nr})}\lesssim \sup_{|\alpha|\leq [3nr/2]+1} C_{\alpha'}\cdot C_\alpha<\infty,
		\end{equation} where $\tilde{a}(x,\xi):=\tilde{a}_x(\xi).$ Finally, by Theorem \ref{HorCondition} the periodic multilinear pseudo-differential operator $T_{\tilde{a}},$ extends to a bounded operator from $L^{p_1}(\mathbb{T}^n)\times L^{p_2}(\mathbb{T}^n)\times \cdots \times L^{p_r}(\mathbb{T}^n)$ into $L^p( \mathbb{T}^n)$ provided that 
		\begin{equation}
		\frac{1}{p}=\frac{1}{p_1}+\cdots+ \frac{1}{p_r}.
		\end{equation} Thus, by using the equality of periodic multilinear pseudo-differential operators $T_{\tilde{a}}=T_{m}$ we conclude the proof. 
	\end{proof}
	
	\subsection{Boundedness of multilinear Fourier integral operators vs boundedness of periodic Fourier integral operators}

	In this subsection we study the relation between the boundedness of multilinear Fourier integral operators and the boundedness of periodic multilinear Fourier integral operators. Our starting point is the following lemma (see Lemma 3.2 in \cite{CarMessiSeno} ).
	\begin{lemma}\label{lemma1} Suppose $f$ is a continuous periodic function on $\mathbb{R}^n$ and let $\{g_{m}\}_m$ be a sequence of uniformly bounded continuous periodic functions on $\mathbb{R}^n.$ If $\{g_{m}\}_m$ converges pointwise to a function $g$ defined on $\mathbb{R}^n$ and $\{ \epsilon_m\}_m$ is a positive sequence of real numbers, then
		\begin{equation}
		\lim_{m\rightarrow\infty}\epsilon_m^{\frac{n}{2}}\int_{\mathbb{R}^n}e^{-\epsilon_m|x|^2}f(x)g_{m}(x)dx=\int_{\mathbb{T}^n}f(x)g(x)dx
		\end{equation} provided that $\epsilon_{m}\rightarrow 0.$
	\end{lemma}
	The following theorem is a multilinear version of Theorem 3.1 in \cite{CarMessiSeno} which in turn is a generalization of a classical multiplier theorem by Stein and Weiss.
	\begin{theorem}\label{teorema principal1}
		Let $1< p<\infty$ and let $\phi$ be a real valued continuous function defined on $\mathbb{T}^n\times\mathbb{R}^{nr}.$ Let us assume that  $a:\mathbb{T}^n\times\mathbb{R}^{nr}\rightarrow \mathbb{C}$ is a continuous bounded function and the multilinear Fourier integral operator
		\begin{equation}
		Tf(x)=\int_{\mathbb{R}^{nr}}e^{i\phi(x,\xi_1,\xi_2,\cdots, \xi_r)}a(x,\xi_1,\xi_2,\cdots, \xi_r)\widehat{f}_1(\xi_1)\cdots \widehat{f}_r(\xi_r)d\xi
		\end{equation} extends to a bounded multilinear operator from $L^{p_1}(\mathbb{R}^n)\times L^{p_2}(\mathbb{R}^n)\times\cdots\times L^{p_r}(\mathbb{R}^n)$ into $L^p(\mathbb{R}^n).$ Then the periodic multilinear Fourier integral operator
		\begin{equation}\label{cp}
		Af(x):=\sum_{\xi\in\mathbb{Z}^{nr}}e^{i\phi(x,\xi_1,\xi_2,\cdots,\xi_r)}a(x,\xi_1,\xi_2,\cdots,\xi_r)(\mathscr{F}_{\mathbb{T}^n}{f}_{1})(\xi_1)\cdots (\mathscr{F}_{\mathbb{T}^n}{f}_{r})(\xi_r)
		\end{equation} also extends to a  bounded multilinear operator from $L^{p_1}(\mathbb{T}^n)\times L^{p_2}(\mathbb{T}^n)\times\cdots \times L^{p_r}(\mathbb{T}^n)$ into $L^p(\mathbb{T}^n),$ provided that 
		\begin{equation*}
		\frac{1}{p_1}+\cdots+\frac{1}{p_r}=\frac{1}{p},\,\,1\leq p_i<\infty.
		\end{equation*}
		Moreover, there exists a positive constant  $C_p$ such that  $$\Vert A \Vert_{\mathscr{B}(L^{p_1}(\mathbb{T}^n)\times L^{p_2}(\mathbb{T}^n)\times\cdots\times L^{p_r}(\mathbb{T}^n),L^p(\mathbb{T}^n))}\leq C_p\Vert T\Vert_{\mathscr{B}(L^{p_1}(\mathbb{R}^n)\times L^{p_2}(\mathbb{R}^n)\times\cdots\times L^{p_r}(\mathbb{R}^n),L^p(\mathbb{R}^n))}.$$
	\end{theorem}
	
	\begin{proof}
		For the proof, we first prove the following identity for trigonometric polynomials $P_i$ and $Q$ on $\mathbb{T}^n$
		
		\begin{align}\label{eq24}
		&\lim_{\varepsilon\rightarrow 0}\varepsilon^{n/2}     \int_{\mathbb{R}^n}(T (P_1 w_{\alpha_1\varepsilon},P_2w_{\alpha_2\varepsilon},\cdots,P_rw_{\alpha_r\varepsilon}))(x)\overline{Q(x)}w_{\varepsilon \beta}(x)dx \nonumber\\
		&=c_{n,r,p}\int_{\mathbb{T}^n}A(P_1,P_2,\cdots,P_r)\overline{Q(x)}dx,\,\,w_{\delta}(x)=e^{-\delta|x|^2},\delta>0,
		\end{align} for some positive constant $c_{n,r,p}>0.$ We will assume that $$\sum_{j=1}^r\alpha_j+\beta=1,\,\,\,\alpha_i,\beta>0.$$
		By linearity we only need to prove \eqref{eq24} when $P_i(x_i)=e^{i2\pi m_i x_i}$ and $Q(x)=e^{i2\pi k x}$ for $k$ and $m_i$ in $\mathbb{Z}^n,$  $1\leq i\leq r.$ The right hand side of \eqref{eq24} can be computed as follows,
		\begin{align*}
		&\int_{\mathbb{T}^n}(A (P_1,P_2,\cdots,P_r))(x)\overline{Q(x)}dx\\ &=\int_{\mathbb{T}^n}\left(\sum_{\xi\in \mathbb{Z}^{nr}} e^{i\phi(x,\xi)}{a(x,\xi)\delta_{m_1,\xi_1}\delta_{m_2,\xi_2}}\cdots \delta_{m_r,\xi_r} \right)\overline{Q(x)}dx\\
		&=\int_{\mathbb{T}^n} e^{i\phi(x,m_1,m_2,\cdots,m_r)}a(x,m_1,m_2,\cdots,m_r)\overline{Q(x)}dx\\
		&=\int_{\mathbb{T}^n} e^{i\phi(x,m_1,m_2,\cdots,m_r)-i2\pi kx }a(x,m_1,m_2,\cdots,m_r)dx.
		\end{align*}
		Now, we compute the left hand side of \eqref{eq24}. Taking under consideration that the euclidean Fourier transform of $P_i(x)w_{\alpha_i\varepsilon}$ is given by
		\begin{equation}
		\mathscr{F}_{\mathbb{R}^n}(P_jw_{\alpha_j\varepsilon})(\xi_j)=(\alpha_j\varepsilon)^{-\frac{n}{2}}e^{-|\xi_j-m_j|^2/\alpha_j\varepsilon},
		\end{equation}
		by the Fubini theorem we have 
		\begin{align*}
		& \int_{\mathbb{R}^n}(T (P_1 w_{\alpha_1\varepsilon},P_2w_{\alpha_2\varepsilon},\cdots,P_rw_{\alpha_r\varepsilon}))(x)\overline{Q(x)}w_{\varepsilon \beta}(x)dx \\ &=\int_{\mathbb{R}^n}\int_{\mathbb{R}^{nr}}e^{i\phi( x,\xi)}a(x,\xi) \prod_{j=1}^r(\alpha_j\varepsilon)^{-\frac{n}{2}}e^{-|\xi_j-m_j|^2/\alpha_j\varepsilon} \overline{Q(x)}w_{\varepsilon \beta}(x)d\xi dx\\
		&=\int_{\mathbb{R}^n}\int_{\mathbb{R}^{nr}}e^{i\phi( x,\xi)-i2\pi kx} e^{-\pi\varepsilon\beta|x|^2} a(x,\xi) dx \prod_{j=1}^r(\alpha_j\varepsilon)^{-\frac{n}{2}}e^{-|\xi_j-m_j|^2/\alpha_j\varepsilon} d\xi\\
		&=\int_{\mathbb{R}^n}\int_{\mathbb{R}^{nr}}e^{i\phi( x,(\alpha_1\varepsilon)^{\frac{1}{2}}\eta_1+m_1,\cdots, \alpha_r\varepsilon)^{\frac{1}{2}}\eta_r+m_r )-i2\pi kx}\\
		&\hspace{4cm}a( x,(\alpha_1\varepsilon)^{\frac{1}{2}}\eta_1+m_1,\cdots, (\alpha_r\varepsilon)^{\frac{1}{2}}\eta_r+m_r) e^{-\pi\varepsilon\beta|x|^2}dx  e^{-|\eta|^2} d\eta.\\
		\end{align*}
		So, we have
		\begin{align*}
		&\lim_{\varepsilon\rightarrow 0}\varepsilon^{n/2}     \int_{\mathbb{R}^n}(T (P_1 w_{\alpha_1\varepsilon},P_2w_{\alpha_2\varepsilon},\cdots,P_rw_{\alpha_r\varepsilon}))(x)\overline{Q(x)}w_{\varepsilon \beta}(x)dx\\
		&=\lim_{\varepsilon\rightarrow 0}\beta^{-\frac{n}{2}}(\beta\varepsilon)^{n/2}     \int_{\mathbb{R}^n}(T (P_1 w_{\alpha_1\varepsilon},P_2w_{\alpha_2\varepsilon},\cdots,P_rw_{\alpha_r\varepsilon}))(x)\overline{Q(x)}w_{\varepsilon \beta}(x)dx\\
		&=\lim_{\varepsilon\rightarrow 0}\beta^{-\frac{n}{2}}(\beta\varepsilon)^{n/2} \int_{\mathbb{R}^n}\int_{\mathbb{R}^{nr}}e^{i\phi( x,(\alpha_1\varepsilon)^{\frac{1}{2}}\eta_1+m_1,\cdots, \alpha_r\varepsilon)^{\frac{1}{2}}\eta_r+m_r )-i2\pi kx}\\
		&\hspace{4cm}a( x,(\alpha_1\varepsilon)^{\frac{1}{2}}\eta_1+m_1,\cdots, (\alpha_r\varepsilon)^{\frac{1}{2}}\eta_r+m_r) e^{-\pi\varepsilon\beta|x|^2}dx  e^{-|\eta|^2} d\eta.
		\end{align*}
		By Lemma \ref{lemma1}, we have
		\begin{align*}
		\lim_{\varepsilon\rightarrow 0} (\beta\varepsilon)^{n/2} &\int_{\mathbb{R}^n}\int_{\mathbb{R}^{nr}}e^{i\phi( x,(\alpha_1\varepsilon)^{\frac{1}{2}}\eta_1+m_1,\cdots, \alpha_r\varepsilon)^{\frac{1}{2}}\eta_r+m_r )-i2\pi kx}\\
		&\hspace{2cm}a( x,(\alpha_1\varepsilon)^{\frac{1}{2}}\eta_1+m_1,\cdots, (\alpha_r\varepsilon)^{\frac{1}{2}}\eta_r+m_r) e^{-\pi\varepsilon\beta|x|^2}dx  e^{-|\eta|^2} d\eta\\
		&=\int_{\mathbb{T}^n}e^{i\phi( x,m_1,m_2,\cdots,m_r)-i2\pi kx}a(x,m_1,m_2,\cdots,m_r)dx.
		\end{align*}
		Taking into account that $\int_{\mathbb{R}^{nr}}e^{-|\eta|^2}d\eta=\pi^{nr/2},$ and that $a$ is a  continuous bounded function, by the dominated convergence theorem we have
		\begin{align*}
		&\lim_{\varepsilon\rightarrow 0}\varepsilon^{n/2}     \int_{\mathbb{R}^n}(T (P_1 w_{\alpha_1\varepsilon},P_2w_{\alpha_2\varepsilon},\cdots,P_rw_{\alpha_r\varepsilon}))(x)\overline{Q(x)}w_{\varepsilon \beta}(x)dx\\
		&=\beta^{-n/2}\pi^{nr/2}\int_{\mathbb{T}^n}e^{i\phi( x,m)-i2\pi kx}a(x,m)dx,\,m=(m_1,\cdots,m_r).
		\end{align*}
		If we assume that $T$ is a bounded multilinear  operator from $L^{p_1}(\mathbb{R}^n)\times L^{p_2}(\mathbb{R}^n)\times\cdots\times  L^{p_r}(\mathbb{R}^n)$ into $L^p(\mathbb{R}^n),$ then the restriction of $A$ to trigonometric polynomials is a bounded operator  from $L^{p_1}(\mathbb{R}^n)\times L^{p_2}(\mathbb{R}^n)\times\cdots\times  L^{p_r}(\mathbb{R}^n)$ into $L^p(\mathbb{R}^n).$ In fact, if $\alpha_i=\frac{1}{p_i}$ and $\beta=\frac{1}{p'}$ we obtain with $c_{n,r,p}=\beta^{n/2}\pi^{-nr/2},$
		\begin{align*} 
		&\Vert A(P_1,P_2,\cdots,P_r)\Vert_{L^p(\mathbb{T}^n)} =\sup_{\Vert Q \Vert_{L^{p'}(\mathbb{T}^n)}=1}\left|\int_{\mathbb{T}^n}(A(P_1,P_2,\cdots,P_r))(x)\overline{Q(x)}dx \right|\\
		&=\sup_{\Vert Q \Vert_{L^{p'}(\mathbb{T}^n)}=1}\lim_{\varepsilon\rightarrow 0} \varepsilon^{n/2} c_{n,r,p}\left| \int_{\mathbb{R}^n}(T (P_1 w_{\alpha_1\varepsilon},P_2w_{\alpha_2\varepsilon},\cdots,P_rw_{\alpha_r\varepsilon}))(x)\overline{Q(x)}w_{\varepsilon \beta}(x)dx \right|\\
		&\leq\sup_{\Vert Q \Vert_{L^{p'}(\mathbb{T}^n)}=1}\lim_{\varepsilon\rightarrow 0} \varepsilon^{n/2}c_{n,r,p}     \Vert T \Vert_{\mathscr{B}(L^{p_1}\times\cdots\times L^{p_r},L^p)}\\
		&\hspace{2cm}\prod_{j=1}^n\Vert P_jw_{\varepsilon/p_j}\Vert_{L^{p_j}(\mathbb{R}^n)}\Vert Qw_{\varepsilon/p'}\Vert_{L^{p'}(\mathbb{T}^n)}\\
		&\leq\sup_{\Vert Q \Vert_{L^{p'}(\mathbb{T}^n)}=1} \Vert T \Vert_{\mathscr{B}(L^{p_1}\times\cdots\times L^{p_r},L^p)}    \lim_{\varepsilon\rightarrow 0}  c_{n,r,p} \prod_{j=1}^{n}\left(    \varepsilon^{n/2}\int_{\mathbb{R}^n}|P_j(x)|^pe^{-\pi\varepsilon|x|^2}dx\right)^{\frac{1}{p_j}}\\  
		&\hspace{8cm} \times\left(  \varepsilon^{n/2}  \int_{\mathbb{R}^n}|Q(x)|^{p'}e^{-\pi\varepsilon|x|^2}dx   \right)^{\frac{1}{p'}}\\
		&=\sup_{\Vert Q \Vert_{L^{p'}(\mathbb{T}^n)}=1} \Vert T \Vert_{\mathscr{B}(L^{p_1}\times\cdots\times L^{p_r},L^p)} c_{n,r,p}\left(  \int_{\mathbb{T}^n}|P_j(x)|^{p_j}dx\right)^{\frac{1}{p_j}}  \left(   \int_{\mathbb{T}^n}|Q(x)|^{p'}dx   \right)^{\frac{1}{p'}}\\
		&= c_{n,r,p}\Vert T \Vert_{\mathscr{B}(L^{p_1}\times\cdots\times L^{p_r},L^p)} \prod_{j=1}^n \Vert P_j\Vert_{L^{p_j}(\mathbb{T}^n)}.
		\end{align*}
		Because the restriction of the periodic multilinear operator $A$ to trigonometric polynomials in every variable is a bounded multilinear operator from $L^{p_1}(\mathbb{R}^n)\times L^{p_2}(\mathbb{R}^n)\times \cdots\times L^{p_r}(\mathbb{R}^n)$ into $L^p(\mathbb{R}^n),$ this restriction admits a unique bounded multilineal extension $L^{p_1}(\mathbb{R}^n)\times L^{p_2}(\mathbb{R}^n)\times\cdots\times  L^{p_r}(\mathbb{R}^n)$ into $L^p(\mathbb{R}^n).$ So, we finish the proof.
	\end{proof}

	Now, we consider the following special case, but, we remove the previous condition on the periodicity of the phase function by considering the particular one $\phi(x,\xi)=(x,\xi):=x\cdot\xi.$
	
	\begin{theorem}\label{teorema principal2}
		Let $1< p<\infty$  and let  $a:\mathbb{R}^{nr}\rightarrow \mathbb{C}$ be a continuous bounded function. Let us assume that the multilinear Fourier multiplier operator
		\begin{equation}
		Tf(x)=\int_{\mathbb{R}^{nr}}e^{i2\pi(x,\xi_1,\xi_2,\cdots ,\xi_r)}a(\xi_1,\xi_2,\cdots \xi_r)\widehat{f}_1(\xi_1)\cdots \widehat{f}_r(\xi_r)d\xi
		\end{equation} extends to a bounded multilinear operator from $L^{p_1}(\mathbb{R}^n)\times L^{p_2}(\mathbb{R}^n)\times\cdots \times L^{p_r}(\mathbb{R}^n)$ into $L^p(\mathbb{R}^n).$ Then the periodic multilinear Fourier multiplier 
		\begin{equation}\label{cp}
		Af(x):=\sum_{\xi\in\mathbb{Z}^{nr}}e^{i2\pi(x, \xi_1,\xi_2,\cdots,\xi_r)}a(\xi_1,\xi_2,\cdots,\xi_r)(\mathscr{F}_{\mathbb{T}^n}{f}_{1})(\xi_1)\cdots (\mathscr{F}_{\mathbb{T}^n}{f}_{r})(\xi_r)
		\end{equation} also extends to a  bounded multilinear operator from $L^{p_1}(\mathbb{T}^n)\times L^{p_2}(\mathbb{T}^n)\times\cdots\times L^{p_r}(\mathbb{T}^n)$ into $L^p(\mathbb{T}^n),$ provided that 
		\begin{equation*}
		\frac{1}{p_1}+\cdots+\frac{1}{p_r}=\frac{1}{p},\,\,1\leq p_i<\infty.
		\end{equation*}
		Moreover, there exists a positive constant $C_p$ such that the following  inequality holds. $$\Vert A \Vert_{\mathscr{B}(L^{p_1}(\mathbb{T}^n)\times L^{p_2}(\mathbb{T}^n)\times\cdots \times L^{p_r}(\mathbb{T}^n),L^p(\mathbb{T}^n))}\leq C_p\Vert T\Vert_{\mathscr{B}(L^{p_1}(\mathbb{R}^n)\times L^{p_2}(\mathbb{R}^n)\times \cdots \times L^{p_r}(\mathbb{R}^n),L^p(\mathbb{R}^n))}.$$
	\end{theorem}
	
	\begin{proof}
		Similar to  our previous result for Fourier integral operators, for the proof of this theorem also we first prove the following identity for trigonometric polynomials $P_i$ and $Q$ on $\mathbb{T}^n$
		\begin{align}\label{eq1}
		&\lim_{\varepsilon\rightarrow 0}\varepsilon^{n/2}     \int_{\mathbb{R}^n}(T (P_1 w_{\alpha_1\varepsilon},P_2w_{\alpha_2\varepsilon},\cdots,P_rw_{\alpha_r\varepsilon}))(x)\overline{Q(x)}w_{\varepsilon \beta}(x)dx \nonumber\\
		&=c_{n,r,p}\int_{\mathbb{T}^n}A(P_1,P_2,\cdots,P_r)\overline{Q(x)}dx,\,\,w_{\delta}(x)=e^{-\delta|x|^2},\delta>0,
		\end{align}  for some positive constant $c_{n,r,p}>0.$ We will assume that $$\sum_{j=1}^n\alpha_j+\beta=1,\alpha_i,\beta>0.$$
		By linearity we only need to prove \eqref{eq1} when $P_i(x_i)=e^{i2\pi m_i x_i}$ and $Q(x)=e^{i2\pi k x}$ for $k$ and $m_i$ in $\mathbb{Z}^n,$  $1\leq i\leq r.$ The right hand side of \eqref{eq1} can be computed as follows,
		\begin{align*}
		&\int_{\mathbb{T}^n}(A (P_1,P_2,\cdots,P_r))(x)\overline{Q(x)}dx\\ &=\int_{\mathbb{T}^n}\left(\sum_{\xi\in \mathbb{Z}^{nr}} e^{i2\pi(x,\xi)}{a(\xi)\delta_{m_1,\xi_1}\delta_{m_2,\xi_2}}\cdots \delta_{m_r,\xi_r} \right)\overline{Q(x)}dx\\
		&=\int_{\mathbb{T}^n} e^{i2\pi(x,m_1,m_2,\cdots,m_r)}a(m_1,m_2,\cdots,m_r)\overline{Q(x)}dx\\
		&=\int_{\mathbb{T}^n} e^{i2\pi(x,m_1,m_2,\cdots,m_r)-i2\pi kx }a(m_1,m_2,\cdots,m_r)dx\\
		&=a(m_1,m_2,\cdots,m_r)\delta_{k,\sum_{j=1}^rm_j}.
		\end{align*}
		Now, we compute the left hand side of \eqref{eq1}. Taking under consideration that the euclidean Fourier transform of $P_i(x)w_{\alpha_i\varepsilon}$ is given by
		\begin{equation}
		\mathscr{F}_{\mathbb{R}^n}(P_jw_{\alpha_j\varepsilon})(\xi_j)=(\alpha_j\varepsilon)^{-\frac{n}{2}}e^{-|\xi_j-m_j|^2/\alpha_j\varepsilon},
		\end{equation}
		by the Fubini theorem we have 
		\begin{align*}
		& \int_{\mathbb{R}^n}(T (P_1 w_{\alpha_1\varepsilon},P_2w_{\alpha_2\varepsilon},\cdots,P_rw_{\alpha_r\varepsilon}))(x)\overline{Q(x)}w_{\varepsilon \beta}(x)dx \\ &=\int_{\mathbb{R}^n}\int_{\mathbb{R}^{nr}}e^{i2\pi(x, \xi)}a(\xi) \prod_{j=1}^r(\alpha_j\varepsilon)^{-\frac{n}{2}}e^{-|\xi_j-m_j|^2/\alpha_j\varepsilon} \overline{Q(x)}w_{\varepsilon \beta}(x)d\xi dx\\
		&=\int_{\mathbb{R}^n}\int_{\mathbb{R}^{nr}}e^{i2\pi( x,\xi)-i2\pi kx} e^{-\pi\varepsilon\beta|x|^2} a(\xi) dx \prod_{j=1}^r(\alpha_j\varepsilon)^{-\frac{n}{2}}e^{-|\xi_j-m_j|^2/\alpha_j\varepsilon} d\xi\\
		&=\int_{\mathbb{R}^n}\int_{\mathbb{R}^{nr}}e^{i2\pi( (\alpha_1\varepsilon)^{\frac{1}{2}}\eta_1+m_1,\cdots, \alpha_r\varepsilon)^{\frac{1}{2}}\eta_r+m_r )-i2\pi kx}\\
		&\hspace{4cm}a( (\alpha_1\varepsilon)^{\frac{1}{2}}\eta_1+m_1,\cdots, (\alpha_r\varepsilon)^{\frac{1}{2}}\eta_r+m_r) e^{-\pi\varepsilon\beta|x|^2}dx  e^{-|\eta|^2} d\eta.\\
		\end{align*}
		So, we have
		\begin{align*}
		&\lim_{\varepsilon\rightarrow 0}\varepsilon^{n/2}     \int_{\mathbb{R}^n}(T (P_1 w_{\alpha_1\varepsilon},P_2w_{\alpha_2\varepsilon},\cdots,P_rw_{\alpha_r\varepsilon}))(x)\overline{Q(x)}w_{\varepsilon \beta}(x)dx\\
		&=\lim_{\varepsilon\rightarrow 0}\beta^{-\frac{n}{2}}(\beta\varepsilon)^{n/2}     \int_{\mathbb{R}^n}(T (P_1 w_{\alpha_1\varepsilon},P_2w_{\alpha_2\varepsilon},\cdots,P_rw_{\alpha_r\varepsilon}))(x)\overline{Q(x)}w_{\varepsilon \beta}(x)dx\\
		&=\lim_{\varepsilon\rightarrow 0}\beta^{-\frac{n}{2}}(\beta\varepsilon)^{n/2} \int_{\mathbb{R}^n}\int_{\mathbb{R}^{nr}}e^{i2\pi( x,(\alpha_1\varepsilon)^{\frac{1}{2}}\eta_1+m_1,\cdots, \alpha_r\varepsilon)^{\frac{1}{2}}\eta_r+m_r )-i2\pi kx}\\
		&\hspace{4cm}a( (\alpha_1\varepsilon)^{\frac{1}{2}}\eta_1+m_1,\cdots, (\alpha_r\varepsilon)^{\frac{1}{2}}\eta_r+m_r) e^{-\pi\varepsilon\beta|x|^2}dx  e^{-|\eta|^2} d\eta.
		\end{align*}
		By Lemma \eqref{lemma1}, we have
		\begin{align*}
		\lim_{\varepsilon\rightarrow 0} (\beta\varepsilon)^{n/2} &\int_{\mathbb{R}^n}\int_{\mathbb{R}^{nr}}e^{i2\pi( x,(\alpha_1\varepsilon)^{\frac{1}{2}}\eta_1+m_1,\cdots, \alpha_r\varepsilon)^{\frac{1}{2}}\eta_r+m_r )-i2\pi kx}\\
		&\hspace{2cm}a( (\alpha_1\varepsilon)^{\frac{1}{2}}\eta_1+m_1,\cdots, (\alpha_r\varepsilon)^{\frac{1}{2}}\eta_r+m_r) e^{-\pi\varepsilon\beta|x|^2}dx  e^{-|\eta|^2} d\eta\\
		&=\int_{\mathbb{T}^n}e^{i2\pi( x,m_1,m_2,\cdots,m_r)-i2\pi kx}a(m_1,m_2,\cdots,m_r)dx\int_{\mathbb{R}^{nr}} e^{-|\eta|^2}d\eta,\\
		&=a(m_1,m_2,\cdots,m_r)\delta_{k, \sum_{j=1}^rm_j  } \int_{\mathbb{R}^{nr}} e^{-|\eta|^2}d\eta  .
		\end{align*}
		Taking into account that $\int_{\mathbb{R}^{nr}}e^{-|\eta|^2}d\eta=\pi^{nr/2},$ and that $a$ is a  continuous bounded function, by the dominated convergence theorem we have
		\begin{align}
		&\lim_{\varepsilon\rightarrow 0}\varepsilon^{n/2}     \int_{\mathbb{R}^n}(T (P_1 w_{\alpha_1\varepsilon},P_2w_{\alpha_2\varepsilon},\cdots,P_rw_{\alpha_r\varepsilon}))(x)\overline{Q(x)}w_{\varepsilon \beta}(x)dx\\
		&=(1/\beta)^{n/2}\pi^{nr/2}a(m_1,m_2,\cdots,m_r)\delta_{k,  \sum_{j=1}^rm_j }.
		\end{align}
		If we assume that $T$ is a bounded multilinear  operator from $L^{p_1}(\mathbb{R}^n)\times L^{p_2}(\mathbb{R}^n)\times\cdots\times  L^{p_r}(\mathbb{R}^n)$ into $L^p(\mathbb{R}^n),$ then the restriction of $A$ to trigonometric polynomials is a bounded operators on from $L^{p_1}(\mathbb{R}^n)\times L^{p_2}(\mathbb{R}^n)\times\cdots\times  L^{p_r}(\mathbb{R}^n)$ into $L^p(\mathbb{R}^n).$ In fact, if $\alpha_i=\frac{1}{p_i}$ and $\beta=\frac{1}{p'}$ we obtain for  $c_{n,r,p}=\beta^{n/2}\pi^{-nr/2},$
		\begin{align*}
		&\Vert A(P_1,P_2,\cdots,P_r)\Vert_{L^p(\mathbb{T}^n)} =\sup_{\Vert Q \Vert_{L^{p'}(\mathbb{T}^n)}=1}\left|\int_{\mathbb{T}^n}(A(P_1,P_2,\cdots,P_r))(x)\overline{Q(x)}dx \right|\\
		&=\sup_{\Vert Q \Vert_{L^{p'}(\mathbb{T}^n)}=1}\lim_{\varepsilon\rightarrow 0} \varepsilon^{n/2} c_{n,r,p}\left| \int_{\mathbb{R}^n}(T (P_1 w_{\alpha_1\varepsilon},P_2w_{\alpha_2\varepsilon},\cdots,P_rw_{\alpha_r\varepsilon}))(x)\overline{Q(x)}w_{\varepsilon \beta}(x)dx \right|\\
		&\leq\sup_{\Vert Q \Vert_{L^{p'}(\mathbb{T}^n)}=1}\lim_{\varepsilon\rightarrow 0} \varepsilon^{n/2}  c_{n,p,r}\Vert T \Vert_{\mathscr{B}(L^{p_1}\times\cdots\times L^{p_r},L^p)}\\
		&\hspace{2cm}\prod_{j=1}^n\Vert P_jw_{\varepsilon/p_j}\Vert_{L^{p_j}(\mathbb{R}^n)}\Vert Qw_{\varepsilon/p'}\Vert_{L^{p'}(\mathbb{T}^n)}\\
		&\leq\sup_{\Vert Q \Vert_{L^{p'}(\mathbb{T}^n)}=1} \Vert T \Vert_{\mathscr{B}(L^{p_1}\times\cdots\times L^{p_r},L^p)}    \lim_{\varepsilon\rightarrow 0}  c_{n,r,p} \prod_{j=1}^{n}\left(    \varepsilon^{n/2}\int_{\mathbb{R}^n}|P_j(x)|^pe^{-\pi\varepsilon|x|^2}dx\right)^{\frac{1}{p_j}}\\  
		&\hspace{8cm} \times\left(  \varepsilon^{n/2}  \int_{\mathbb{R}^n}|Q(x)|^{p'}e^{-\pi\varepsilon|x|^2}dx   \right)^{\frac{1}{p'}}\\
		&=\sup_{\Vert Q \Vert_{L^{p'}(\mathbb{T}^n)}=1} \Vert T \Vert_{\mathscr{B}(L^{p_1}\times\cdots\times L^{p_r},L^p)}c_{n,r,p}\prod_{j=1}^{n}\left(  \int_{\mathbb{T}^n}|P_j(x)|^{p_j}dx\right)^{\frac{1}{p_j}}  \left(   \int_{\mathbb{T}^n}|Q(x)|^{p'}dx   \right)^{\frac{1}{p'}}\\
		&= \Vert T \Vert_{\mathscr{B}(L^{p_1}\times\cdots\times L^{p_r},L^p)} c_{n,r,p}\prod_{j=1}^n \Vert P_j\Vert_{L^{p_j}(\mathbb{T}^n)}.
		\end{align*}
		Because the restriction of the periodic multilinear operator $A$ to trigonometric polynomials in every variable is a bounded multilinear operator from $L^{p_1}(\mathbb{R}^n)\times L^{p_2}(\mathbb{R}^n)\times \cdots\times L^{p_r}(\mathbb{R}^n)$ into $L^p(\mathbb{R}^n),$ this restriction admits a unique bounded multilineal extension $L^{p_1}(\mathbb{R}^n)\times L^{p_2}(\mathbb{R}^n)\times\cdots\times  L^{p_r}(\mathbb{R}^n)$ into $L^p(\mathbb{R}^n).$ So, we finish the proof.
	\end{proof}
	
	\begin{theorem} Let $T_m$ be a periodic multilinear Fourier multiplier. Let us assume that the symbol $m$ satisfies the estimates
		
		$$|\Delta_{\xi_1}^{\alpha_{1} }\cdots \Delta_{\xi_r}^{\alpha_{r}  }m(\xi_1,\xi_2,\cdots,\xi_r)|\leq C_\alpha \langle \xi \rangle^{-|\alpha_1|-\cdots -|\alpha_r|},\,\,\,|\alpha|\leq [\frac{nr}{2}]+1.$$
		Then the operator $T_m$ extends to a  bounded multilinear operator from $L^{p_1}(\mathbb{T}^n)\times L^{p_2}(\mathbb{T}^n)\times\cdots \times L^{p_r}(\mathbb{T}^n)$ into $L^p(\mathbb{T}^n),$ provided that 
		\begin{equation*}
		\frac{1}{p_1}+\cdots+\frac{1}{p_r}=\frac{1}{p},\,\,1\leq p_i<\infty.
		\end{equation*}
		
	\end{theorem}
	\begin{proof} Let us assume that the multilinear symbol $m$ satisfies:
		$$\Vert m\Vert_{l.u.,H^s(\mathbb{R}^{nr})}:=\sup_{k\in\mathbb{Z}}\Vert m(2^{k}\eta_1
		,2^k\eta_2,\cdots ,2^k\eta_r)\phi\Vert_{H^s}<\infty,$$ $\phi\in \mathscr{D}(0,\infty),$  $s>nr/2.$
		From Tomita multilinear theorem \cite{Tomitabilinear}, this condition implies the boundedness of a multilinear Fourier multiplier  associated with $m$ from $L^{p_1}(\mathbb{R}^n)\times L^{p_2}(\mathbb{R}^n)\times\cdots \times L^{p_r}(\mathbb{R}^n)$ into $L^p(\mathbb{R}^n),$ provided that 
		\begin{equation*}
		\frac{1}{p_1}+\cdots+\frac{1}{p_r}=\frac{1}{p},\,\,1\leq p_i<\infty.
		\end{equation*} So, by applying Theorem \ref{teorema principal2} we deduce the boundedness of the periodic multilinear operator $T_m$ from $L^{p_1}(\mathbb{T}^n)\times L^{p_2}(\mathbb{T}^n)\times\cdots \times L^{p_r}(\mathbb{T}^n)$ into $L^p(\mathbb{T}^n).$  Now,  let us assume that $m$ satisfies \begin{equation*}
		|\Delta_{\xi_1}^{\alpha_1}\Delta_{\xi_2}^{\alpha_2}\cdots \Delta_{\xi_r}^{\alpha_r}m(\xi_1,\cdots,\xi_r)|\leq C_\alpha\langle \xi\rangle^{-|\alpha_1|-\cdots -|\alpha_r|},
		\end{equation*} for all $|\alpha|:=  |\alpha_1|+\cdots +|\alpha_r|\leq [nr/2]+1.$ 
		Now, by using Lemma \ref{IClases}, there exists $\tilde{a}$ such that $m=\tilde{a}|_{  {\mathbb{Z}}^{nr}}$ and 
		\begin{equation*}
		|\partial^\alpha_\xi \tilde{a}(\xi)|\leq C_{\alpha'}\cdot C_\alpha\langle \xi\rangle^{-|\alpha|},\,\,\,|\alpha|\leq [nr/2]+1.
		\end{equation*}
		Now, taking into account that
		\begin{equation*}
		\Vert \tilde{a}\Vert_{l.u.,H^s_{loc}(\mathbb{R}^{nr})}\lesssim \sup_{\xi\in \mathbb{R}^n}\sup_{|\alpha|\leq [nr/2]+1}\langle \xi\rangle^{|\alpha|}|\partial^\alpha_\xi \tilde{a}(\xi)|\leq \sup_{|\alpha|\leq [nr/2]+1} C_{\alpha'}\cdot C_\alpha,
		\end{equation*} for $\frac{nr}{2}<s<[nr/2]+1,$ we deduce
		\begin{equation*}
		\Vert \tilde{a}(\xi)\Vert_{l.u.,H^s_{loc}(\mathbb{R}^n_\xi)}\lesssim \sup_{|\alpha|\leq [nr/2]+1} C_{\alpha'}\cdot C_\alpha<\infty.
		\end{equation*}  Finally, by the first part of the proof, the periodic multilinear pseudo-differential operator $T_{\tilde{a}},$ extends to a bounded operator from $L^{p_1}(\mathbb{T}^n)\times L^{p_2}(\mathbb{T}^n)\times \cdots \times L^{p_r}(\mathbb{T}^n)$ into $L^p( \mathbb{T}^n)$ provided that 
		\begin{equation*}
		\frac{1}{p}=\frac{1}{p_1}+\cdots+ \frac{1}{p_r}.
		\end{equation*} Thus, by using the equality of periodic multilinear pseudo-differential operators $T_{\tilde{a}}=T_{m}$ we conclude the proof. 
	\end{proof}

	The previous multilinear result and the Sobolev embedding theorem give the following theorem where we relax the number of derivatives imposed in the multilinear pseudo-differential theorem proved in the preceding subsection.
	
	\begin{theorem}\label{Teo:2} Let $T_m$ be a periodic multilinear pseudo-differential operator. Let us assume that $m$ satisfies toroidal conditions of the type,
		\begin{equation*}
		|\partial_{x}^\beta\Delta_{\xi_1}^{\alpha_{1} }\cdots \Delta_{\xi_r}^{\alpha_{r}  } m(x,\xi_1,\xi_2,\cdots,\xi_r)|\leq C_\alpha \langle \xi \rangle^{-|\alpha_1|-\cdots -|\alpha_r|},
		\end{equation*} where $|\alpha|\leq [\frac{nr}{2}]+1,$ and $|\beta|\leq [\frac{n}{p}]+1.$ Then $T_m$ extends to a  bounded multilinear operator from $L^{p_1}(\mathbb{T}^n)\times L^{p_2}(\mathbb{T}^n)\times\cdots \times L^{p_r}(\mathbb{T}^n)$ into $L^p(\mathbb{T}^n),$ provided that 
		\begin{equation*}
		\frac{1}{p_1}+\cdots+\frac{1}{p_r}=\frac{1}{p},\,\,1\leq p_i<\infty.
		\end{equation*}
		
	\end{theorem} 
	
	\begin{proof}
		
		For every $z\in\mathbb{T}^n,$ let us define the multilinear operator
		\begin{equation*}
		A_{z}f(x):=\sum_{\xi\in\mathbb{Z}^{nr}}e^{i2\pi(x,\xi)}a(z,\xi_1,\cdots,\xi_r)\prod_{j=1}^r(\mathscr{F}_{{\mathbb{T}^n}}f_j)(\xi_j).
		\end{equation*} Taking into account the identity $A_xf(x)=T_mf(x),$ an application of the Sobolev embedding theorem gives,
		\begin{align*}
		\Vert T_mf \Vert^p_{L^p(\mathbb{T}^n)}&=\int_{\mathbb{T}^n}|A_xf(x)|^p dx \leq \int_{\mathbb{T}^n}\sup_{z\in\mathbb{T}^n}|A_zf(x)|^pdx\\
		&\lesssim \sum_{|\beta|\leq [\frac{n}{p}]+1} \int_{\mathbb{T}^n}\int_{\mathbb{T}^n} \vert\partial_z^\beta A_zf(x)\vert^p \, dz\,dx \\
		&=\sum_{|\beta|\leq [\frac{n}{p}]+1} \int_{\mathbb{T}^n}\int_{\mathbb{T}^n} \vert\partial_z^\beta A_zf(x)\vert^p \, dx\,dz  \sum_{|\beta|\leq [\frac{n}{p}]+1}\int_{\mathbb{T}^n}\Vert\partial_z^\beta A_zf\Vert^p_{L^p(\mathbb{T}^n)}dz.
		\end{align*} Taking into account that the multilinear operator $\partial_z^\beta A_z,$ has multilinear symbol $\partial_{z}^{\beta} m(z,\cdot),$ satisfying the estimates,
		\begin{equation*}
		|\Delta_{\xi_1}^{\alpha_{1} }\cdots \Delta_{\xi_r}^{\alpha_{r}  } [\partial_{z}^\beta m](z,\xi_1,\xi_2,\cdots,\xi_r)|\leq C_\alpha \langle \xi \rangle^{-|\alpha_1|-\cdots -|\alpha_r|},
		\end{equation*} where $|\alpha|\leq [\frac{nr}{2}]+1,$ we can deduce the boundedness of every operator, $\partial_z^\beta A_z$ from $L^{p_1}(\mathbb{T}^n)\times L^{p_2}(\mathbb{T}^n)\times\cdots \times L^{p_r}(\mathbb{T}^n)$ into $L^p(\mathbb{T}^n),$ $
		\frac{1}{p_1}+\cdots+\frac{1}{p_r}=\frac{1}{p},\,\,1\leq p_i<\infty,
		$ with operator norm uniformly bounded in $z\in \mathbb{T}^{n}.$ So, we have,
		\begin{align*}
		\Vert T_mf \Vert^p_{L^p(\mathbb{T}^n)}\lesssim \sum_{|\beta|\leq [\frac{n}{p}]+1}\sup_{z\in\mathbb{T}^n}\Vert \partial_z^\beta A_z \Vert^p_{\mathscr{B} (  L^{p_1}\times \cdots\times L^{p_r},L^p  )   }  \Vert f\Vert^p_{L^{p_1}(\mathbb{T}^n)\times \cdots\times L^{p_r}(\mathbb{T}^n )}.
		\end{align*} Thus, we conclude the proof.
	\end{proof}
	 We end our analysis with the following application to PDEs.
	 \begin{remark}[Kato-Ponce inequality on the torus]\label{limitedregularity'} Kato-Ponce inequalities are estimates of the form,
	 \begin{equation}
    \Vert \mathcal{J}^s(f \cdot g)\Vert_{L^r(\mathbb{R}^n)}\lesssim \Vert \mathcal{J}^s f\Vert_{L^{p_1}(\mathbb{R}^n)}\Vert g\Vert_{L^{q_1}(\mathbb{R}^n)}+\Vert  f\Vert_{L^{p_2}(\mathbb{R}^n)}\Vert \mathcal{J}^s g\Vert_{L^{q_2}(\mathbb{R}^n)}
\end{equation}	
	where $\frac{1}{p_1}+\frac{1}{q_1}=\frac{1}{p_2}+\frac{1}{q_2}=\frac{1}{r}.$ Usually,   we can take $\mathcal{J}^s=(\frac{1}{4\pi^2}\Delta_x)^{s/2}$ or $\mathcal{J}^s=(1+\frac{1}{4\pi^2}\Delta_x)^{s/2}$ where $\Delta_x=-\sum_{j=1}^n{\partial_{x_j}^2}$ is the Laplacian on $\mathbb{R}^n$. As it was pointed out in Grafakos \cite{GrafakosPDE's}, Kato and Ponce used this estimate to obtain commutator
estimates for the Bessel operator which in turn they applied to obtain estimates for
the Euler and Navier-Stokes equations.
	
	The aim of this remark is to deduce the periodic Kato-Ponce inequality  \eqref{periodickatoponce} (see Muscalu and Schlag \cite{Muscaluperiodic}) from Theorem \ref{Teoremapseuod'}. For this, we follow Grafakos \cite{GrafakosPDE's}. Consider the bilinear operator
	\begin{equation}
	    B_s(f,g)=J^s(f\cdot g),\, f,g\in \mathscr{D}(\mathbb{T}^n),\,\,s>0,
	\end{equation} with $J^s=(\mathcal{L})^\frac{s}{2}$ or $J^s=(I+\mathcal{L})^{\frac{s}{2}}.$ Here,  $\mathcal{L}:=\Delta_{\mathbb{T}^n}=-\frac{1}{4\pi^{2}}(\sum_{j=1}^n\partial_{\theta_j}^2),$  denotes the Laplacian on the torus  $\mathbb{T}^n.$ 
	 \end{remark} If $s>0$ and  $J^s=(\mathcal{L})^\frac{s}{2},$ then the operator $B_s$ has the form
	 \begin{equation}
	     B_s(f,g)(x)=J^s(f\cdot g)(x)=\sum_{\xi\in\mathbb{Z}^n}e^{i2\pi x\cdot\xi}|\xi|^s\mathscr{F}_{\mathbb{T}^n}(f\cdot g)(\xi),\,\,x\in\mathbb{T}^n.
	 \end{equation}
	Since $ f(x)\cdot g(x)=\mathscr{F}^{-1}_{\mathbb{T}^n}[ \mathscr{F}_{\mathbb{T}^n}f\ast  \mathscr{F}_{\mathbb{T}^n}g](x) , $ we have that $\mathscr{F}_{\mathbb{T}^n}(f\cdot g)=\mathscr{F}_{\mathbb{T}^n}f\ast  \mathscr{F}_{\mathbb{T}^n}g,$ and consequently,
	\begin{equation}
	 \sum_{\xi\in\mathbb{Z}^n}e^{i2\pi x\cdot\xi}|\xi|^s\mathscr{F}_{\mathbb{T}^n}(f\cdot g)(\xi)=\sum_{\xi,\eta\in\mathbb{Z}^n}e^{i2\pi x\cdot(\xi+\eta)}|\xi+\eta|^s\mathscr{F}_{\mathbb{T}^n}(f)(\xi) \mathscr{F}_{\mathbb{T}^n}(g)(\eta).  
	\end{equation} We now note that  $B_s$ is a periodic bilinear operator associated to symbol $\sigma(\xi,\eta)=|\xi+\eta|^s.$ If $\phi\in \mathscr{D}(\mathbb{R})$ is compactly supported in $[-2,2]$ with $\phi(t)=1$ for $0\leq t\leq 1,$ we can split the bilinear symbol $\sigma$ as
	\begin{equation}
	    \sigma(\xi,\eta)=\left(\frac{|\xi+\eta|^s}{|\xi|^s}(1-\phi)(|\xi|/|\eta|)   \right)|\xi|^s+\left(  \frac{|\xi+\eta|^s}{|\eta|^s}(\phi)(|\xi|/|\eta )\right)|\eta|^s.
	\end{equation} Putting
	\begin{equation}
	    \tau(\xi,\eta):=\frac{|\xi+\eta|^s}{|\xi|^s}(1-\phi)(|\xi|/|\eta|),\,\,\,\alpha(\xi,\eta):=  \frac{|\xi+\eta|^s}{|\eta|^s}(\phi)(|\xi|/|\eta ),
	\end{equation}
	
	we obtain 
	\begin{equation}
	    B_s(f,g)=T_\tau(J^sf,g)+T_\alpha(f,J^s g).
	\end{equation} Since the bilinear symbols $\tau(\xi,\eta)=\tau(x,\xi,\eta)$ and $\alpha(\xi,\eta)=\alpha(x,\xi,\eta)$ satisfy the hypothesis of Theorem \ref{Teoremapseuod'},
		\begin{equation*}
		\sup_{x\in\mathbb{T}^n}|\Delta_{\xi_1}^{\alpha_1}\Delta_{\xi_2}^{\alpha_2}\cdots \Delta_{\xi_r}^{\alpha_r}m(x,\xi_1,\cdots,\xi_r)|\leq C_{\phi, m,\alpha}\langle \xi\rangle^{-|\alpha|},\,\,\,m=\alpha,\textnormal{  or  }m=\tau,
		\end{equation*} 
	for all $|\alpha|\leq 3n+1$, we obtain the periodic Kato-Ponce inequality,
	\begin{equation}\label{periodickatoponce}
    \Vert J^s( f \cdot g)\Vert_{L^r(\mathbb{T}^n)}\lesssim \Vert J^s f\Vert_{L^{p_1}(\mathbb{T}^n)}\Vert g\Vert_{L^{q_1}(\mathbb{T}^n)}+\Vert  f\Vert_{L^{p_2}(\mathbb{T}^n)}\Vert J^s g\Vert_{L^{q_2}(\mathbb{T}^n)}
\end{equation}	
	where $\frac{1}{p_1}+\frac{1}{q_1}=\frac{1}{p_2}+\frac{1}{q_2}=\frac{1}{r},$  $1<r<\infty,$  $1\leq p_i,q_i\leq \infty.$ A similar analysis can be applied when $J^s=(I+\mathcal{L})^{\frac{s}{2}}.$

	\subsection{Boundedness of multilinear pseudo-differential operators on $\mathbb{Z}^n$}

	In this subsection we provide some results on the boundedness of discrete pseudo-differential operators in the multilinear context. By following Botchway  Kibiti and  Ruzhansky \cite{Ruzhansky}, a symbol $\sigma$ belongs to the discrete H\"ormander class $S^m_{\rho,\delta}(\mathbb{Z}^n\times \mathbb{T}^n),$ if it satisfies discrete inequalities of the type
	\begin{equation*}
	|D_{\xi}^{(\beta)}\Delta_\ell^\alpha \sigma(\ell,\xi)|\leq C_{\alpha,\beta}\langle \ell \rangle^{m-\rho|\alpha|+\delta|\beta|}.
	\end{equation*}
	Taking as starting point this definition, we will investigate the $L^p$ boundedness of discrete multilinear pseudo-differential operators with multilinear symbols satisfying inequalities of the type.
	
	\begin{equation*}
	|\partial_{\xi}^{\beta}\Delta_\ell^\alpha \sigma(\ell,\xi)|\leq C_{\alpha,\beta}\langle \ell \rangle^{m-\rho|\alpha|+\delta|\beta|},\ell\in \mathbb{Z}^{n},\,\,\xi\in\mathbb{T}^{nr}.
	\end{equation*}
	
	Now, we present the following discrete boundedness theorem.
	
	\begin{theorem}
		Let  $\sigma\in L^\infty(\mathbb{Z}^{n}, C^{2\varkappa} (\mathbb{T}^{nr})).$ Let us assume that $\sigma$ satisfies the following discrete inequalitiess
		\begin{equation*}
		|\partial_{\xi}^{\beta} \sigma(\ell,\xi)|\leq C_{\beta},\,\,\ell\in \mathbb{Z}^{n},\,\,\xi\in\mathbb{T}^{nr},\,\sigma(\ell,\xi)=\sigma(\ell) (\xi) ,
		\end{equation*} for all $\beta$ with $|\beta|= 2\varkappa.$ Then $T_\sigma$ extends to a bounded operator from $L^{p_1}(\mathbb{Z}^n)\times L^{p_2}(\mathbb{Z}^n)\times \cdots\times L^{p_r}(\mathbb{Z}^n )$ into $L^{s}(\mathbb{Z}^n )$ provided that $1\leq p_j\leq p \leq\infty,$ and $$ \frac{1}{s}-\frac{1}{p}<\frac{2\varkappa}{nr}-1.$$

	\end{theorem}
	\begin{proof}
		For a compactly supported function $f=(f_1,f_2,\cdots,f_r),$ we have 
		\begin{align*}
		T_\sigma f(\ell) &=\int\limits_{\mathbb{T}^{nr}}e^{i2\pi \ell(\xi_1+\cdots+\xi_r)}\sigma(\ell,\xi)(\mathscr{F}_{\mathbb{Z}^n}{f}_{1})(\xi_1)\cdots (\mathscr{F}_{\mathbb{Z}^n}{f}_{r})(\xi_r)d\xi,\\
		&=\sum_{\ell'\in\mathbb{Z}^{nr}}\left( \int\limits_{\mathbb{T}^{nr}}e^{i2\pi \ell(\xi_1+\cdots+\xi_r)-\ell'\cdot \xi}\sigma(\ell,\xi)d\xi   \right)f_1(\ell_1')\cdots f_r(\ell_r')\\
		&=\sum_{\ell'\in\mathbb{Z}^{nr}}\left( \int\limits_{\mathbb{T}^{nr}}e^{i2\pi( \tilde{\ell}-\ell')\cdot \xi}\sigma(\ell,\xi)d\xi   \right)f_1(\ell_1')\cdots f_r(\ell_r')\\
		&=\sum_{\ell'\in\mathbb{Z}^{nr}}\left( \int\limits_{\mathbb{T}^{nr}}e^{i2\pi( \tilde{\ell}-\ell')\cdot \xi}\sigma(\ell,\xi)d\xi   \right)f_1(\ell_1')\cdots f_r(\ell_r')\\
		&=\sum_{\ell'\in\mathbb{Z}^{nr}}\kappa(\tilde{\ell},\ell')f_1(\ell_1')\cdots f_r(\ell_r'),
		\end{align*}
		where $\kappa(\tilde{\ell},\ell')= \int\limits_{\mathbb{T}^{nr}}e^{i2\pi( \tilde{\ell}-\ell')\cdot \xi}\sigma(\ell,\xi)d\xi   ,$ and $\tilde{\ell}=(\ell,\ell,\cdots,\ell)$ is the notation used in the previous subsection. Now, by applying integration by parts, we get
		\begin{equation*}
		\int\limits_{\mathbb{T}^{nr}}e^{i2\pi( \tilde{\ell}-\ell')\cdot \xi}\sigma(\ell,\xi)d\xi=\langle \tilde{\ell}-\ell' \rangle^{-2\varkappa} \int\limits_{\mathbb{T}^{nr}}e^{i2\pi( \tilde{\ell}-\ell')\cdot \xi}(I-\frac{1}{4\pi^2}\mathcal{L}_{\mathbb{T}^{nr}})^{\varkappa}[\sigma](\ell,\xi)d\xi. 
		\end{equation*}
		Consequently, we have the estimate,
		\begin{equation*}
		|\kappa(\tilde{\ell},\ell')|\leq C_{\varkappa}\langle \tilde{\ell}-\ell' \rangle^{-2\varkappa}.
		\end{equation*} So, if $f(\ell') =f_1(\ell_1')\cdots f_r(\ell_r'),$ for every $\ell\in\mathbb{Z}^n,$
		\begin{align*}
		|T_\sigma f(\ell)|&\lesssim \sum_{\ell'\in\mathbb{Z}^{nr}}\langle \tilde{\ell}-\ell' \rangle^{-2\varkappa}f(\ell')=(\langle \tilde{\ell}-\cdot \rangle^{-2\varkappa}\ast|f|)(\ell),
		\end{align*} where $\ast$ is the convolution on sequences. By the Hausdorff-Young inequality, we have the inequality
		$$ \Vert T_\sigma f \Vert_{L^s(\mathbb{Z}^n)}\leq C\Vert \langle \tilde{\ell}-\cdot \rangle^{-2\varkappa} \Vert_{L^q(\mathbb{Z}^{nr})}\Vert f\Vert_{L^p(\mathbb{Z}^{nr})}=\Vert \langle \ell' \rangle^{-2\varkappa} \Vert_{L^q(\mathbb{Z}^{nr})}\prod_{j=1}^{r}\Vert f_j\Vert_{L^p(\mathbb{Z}^{n})}, $$ for $1+\frac{1}{s}=\frac{1}{p}+\frac{1}{q}.$ We require $2q\varkappa >nr$ or equivalently $q>\frac{nr}{2\varkappa},$ in order that $\Vert\langle \ell \rangle^{-2\varkappa} \Vert_{L^q(\mathbb{Z}^{nr})}<\infty.$ So, by the hypothesis  $\frac{1}{s}-\frac{1}{p}<\frac{2\varkappa}{nr}-1,$ we can take $q=(\frac{1}{s}-\frac{1}{p}+1)^{-1}$ in the previous inequality.
		Now, by taking into account that
		$$ \prod_{j=1}^{r}\Vert f_j\Vert_{L^p(\mathbb{Z}^{n})}\leq \prod_{j=1}^{r}\Vert f_j\Vert_{L^{p_j}(\mathbb{Z}^{n})} ,\,\, p_j\leq p\leq\infty, $$ we finish the proof.
	\end{proof}
	
	As a consequence of the previous result, for $r=1$ and $s=p$  we obtain the following $L^p$-estimate for discrete pseudo-differential operators.
	\begin{corollary}
		Let  $\sigma\in C^{2\varkappa}(\mathbb{Z}^{n}\times \mathbb{T}^{n}).$ Let us assume that $\sigma$ satisfies the following discrete inequalities
		\begin{equation*} \label{derivedcor}
		|\partial_{\xi}^{\beta} \sigma(\ell,\xi)|\leq C_{\beta},\,\,\ell\in \mathbb{Z}^{n},\,\,\xi\in\mathbb{T}^{n},  
		\end{equation*} for all $\beta$ with $|\beta|= 2\varkappa.$ Then $T_\sigma$ extends to a bounded operator from $L^{p}(\mathbb{Z}^n))$ into $L^{p}(\mathbb{Z}^n )$ provided that $1\leq p\leq \infty,$ and $\varkappa>n/2.$
		\end{corollary}
	Let us mention  that Corollary \ref{derivedcor} implies the $L^p$-boundedness of pseudo-differential operators associated to the discrete H\"ormander class $S^0_{0,0}(\mathbb{Z}^n\times \mathbb{T}^n )$ introduced in Botchway, Kibiti and Ruzhansky \cite{Ruzhansky}.  \\
	\\
	
	\section{Nuclear multilinear integral  operators on $L^p(\mu)$-spaces}
	
	We are interested in the nuclearity of multilinear pseudo-differential operators on $\mathbb{Z}^n$ and the torus $\mathbb{T}^n$. The study of the nuclear pseudo-differential operators mainly depends upon a characterization of nuclear operators given by Delgado \cite[Theorem 2.4]{Delgado}. Therefore, it is clear that for studying the multilinear case we need to develop  Delgado's  result in the multilinear setting. In this section we prove a multilinear version of Delgado theorem.   
	
	We recall the definition of a nuclear operator on a Banach space. Let $E$ and $F$ be two Banach spaces. A linear operator $T:E \rightarrow F$ is called {\it $s$-nuclear}, $0<s \leq 1,$ if there exist sequences $\{g_n\}_n \subset E'$ and $\{h_n\}_n \subset F$ such that $ \sum_{n} \|g_n\|_{E'}^s\, \|h_n\|_{F}^s<\infty$ and for every $f \in E,$ we have $$Tf= \sum_{n} \langle f, g_n \rangle h_n.$$
	If $s=1$ then $T$ is called {\it nuclear} operator. 
	
	Let $(X_i,\mu_i),\, 1 \leq i \leq r$ be $\sigma$-finite measure spaces. At times, we write $$\int_{X_1} \int_{X_2} \cdots \int_{X_r}f(x_1,x_2,\cdots,x_r) d\mu(x_1)\, d\mu(x_2), \cdots, d\mu(x_r)$$ as $$\int_{X_1.X_2.\ldots,X_r} f(x)d(\mu_1 \otimes \mu_2\otimes \cdots \otimes \mu_r)(x),$$ where $x= (x_1,x_2, \cdots, x_r) \in X_1 \times X_2 \times \ldots \times X_r.$  We shall also use the notation
	\begin{equation}
	\Vert f\Vert_{L^{p_1}(\mu_1)\times \cdots \times L^{p_r}(\mu_r)  }= \Vert \cdots\Vert  \Vert f(\cdot,x_2,\cdots,x_r)\Vert_{L^{p_1}(\mu_1)}\Vert_{L^{p_2}(\mu_2)}\cdots \Vert_{L^{p_r}(\mu_r)}.
	\end{equation} for the mixed norms. In particular, if $f$ has the form $f(x)=\prod_{j}f_j(x_j),$ then
	$$ \Vert f\Vert_{L^{p_1}(\mu_1)\times \cdots \times L^{p_r}(\mu_r)  }=\prod_{j}\Vert f_j\Vert_{L^{p_j}(\mu_j)}. $$

	We begin with the following important lemma. 
	\begin{lemma} \label{FDel}
		Let $(X_i, \mu_i), 1 \leq i \leq r$ and $(Y, \nu)$ be finite measure spaces. Let $1 \leq p_i,p <\infty, 1 \leq i \leq r$ and let $p_i', q$ be such that $\frac{1}{p_i}+ \frac{1}{p_i'}=1, \frac{1}{p}+\frac{1}{q}=1$ for $1 \leq i \leq r.$ Suppose $f \in L^{p_1}(\mu_1) \times L^{p_2}(\mu_2) \times \cdots \times L^{p_r}(\mu_r),$ where $f= (f_1,f_2,\ldots, f_r)$ such that each $f_i \in L^{p_i}(\mu_i),$ and $\{g_n\}_n$ with $g_n=(g_{n1}, g_{n2}, \ldots, g_{nr})$ and $\{h_n\}_n$ are two sequences in $L^{p_1'}(\mu_1) \times L^{p_2'}(\mu_2) \times \cdots \times L^{p_r'}(\mu_r)$ and $L^p(\nu)$ respectively such that $\sum_{n} \|g_n\|_{L^{p_1'}(\mu_1) \times L^{p_2'}(\mu_2) \times \cdots \times L^{p_r'}(\mu_r)} \|h_n\|_{L^p(\nu)}<\infty.$ Then 
		\begin{itemize}
			\item[(i)] $\lim_{n} \sum_{j=1}^n g_j(x) h_j(y)$ is finite for almost all $(x,y) \in X_1 \times X_2 \times \ldots \times X_r \times Y$ and $\sum_{j=1}^\infty g_j(x) h_j(y)$ is absolutely convergent for almost all $(x,y).$
			\item[(ii)] $k \in L^1(\mu_1 \otimes \mu_2 \otimes \cdots \otimes \mu_r \otimes \nu),$ where $k(x,y)= \sum_{j=1}^\infty g_j(x)\, h_j(y).$
			\item[(iii)] If $k_n(x,y)= \sum_{j=1}^n g_j(x)\, h_j(y)$ then $ \lim_{n} \|k_n-k\|_1=0. $
			\item[(iv)] $$\lim_{n} \int_{X_1,X_2, \ldots,X_r} \left( \sum_{j=1}^n g_j(x)\,h_j(y) \right)f(x) d(\mu)(x) $$ 
			$$ =\int_{X_1,X_2, \ldots,X_r} \left( \sum_{n=1}^\infty g_n(x)\,h_n(y) \right)f(x) d(\mu)(x),$$ where $\mu = \mu_1 \otimes \mu_2 \otimes\cdots \otimes \mu_r$
		\end{itemize}
	\end{lemma}
	\begin{proof}
		Let $k_n(x,y) = \sum_{j=1}^n h_j(y)\, g_j(x)\, f(x),$ where $g_j=(g_{j1},g_{j2},\ldots,g_{jr}) \in L^{p_1'}(\mu_1) \times L^{p_2'}(\mu_2) \times \cdots \times L^{p_r'}(\mu_r)$ and $f=(f_1,f_2,\ldots,f_r) \in L^{p_1}(\mu_1) \times L^{p_2}(\mu_2) \times \cdots \times L^{p_r}(\mu_r).$ Now, by H$\ddot{\text{o}}$lder's inequality, 
		\begin{align*}
		& \int_{X_1,X_2,\ldots,X_r} \int_Y |k_n(x,y)|\, d\nu(y)\, d(\mu_1 \otimes \mu_2 \otimes \cdots \otimes \mu_r)(x) \\
		&\leq \int_{X_1,X_2,\ldots,X_r} \int_Y \sum_{j=1}^n  |h_j(y)\,g_j(x) \, f(x)|\, d\nu(y)\, d(\mu_1 \otimes \mu_2 \otimes \cdots \otimes \mu_r)(x) \\
		& \leq \sum_{j=1}^n \int_Y |h_j(y)|\, d\nu(y) \int_{X_1,X_2, \ldots,X_r} |g_j(x)\,f(x)|\, d(\mu_1 \otimes \mu_2 \otimes \ldots\otimes \mu_r)(x) \\
		& \leq \sum_{j=1}^n \|h_j\|_{L^p(\nu)} (\nu(Y))^{\frac{1}{q}} \|g_j\|_{L^{p_1'}(\mu_1) \times L^{p_2'}(\mu_2) \times \cdots \times L^{p_r'}(\mu_r)} \|f\|_{L^{p_1}(\mu_1) \times \cdots \times L^{p_r}(\mu_r)} \\
		&= (\nu(Y))^{\frac{1}{q}} \|f\|_{L^{p_1}(\mu_1) \times \cdots \times L^{p_r}(\mu_r)} \sum_{j=1}^n\|h_j\|_{L^p(\nu)}\|g_j\|_{L^{p_1'}(\mu_1) \times L^{p_2'}(\mu_2) \times \cdots \times L^{p_r'}(\mu_r)} \\
		& \leq M<\infty\,\,\,\,\text{for all}\,\, n \in \mathbb{N}.
		\end{align*}
		Therefore, $\|k_n\|_1 \leq M$ for all $n \in \mathbb{N}.$ 
		Set $s_n(x,y)= \sum_{j=1}^n |h_j(y) g_j(x)f(x)|.$ Note that $\{s_n\}_n$ is an increasing sequence in $L^1(\mu_1 \otimes \mu_2 \otimes \cdots \otimes \mu_r \otimes \nu)$ and 
		$$\sup_{n} \int_{X_1,X_2,\ldots X_r}\int_Y |s_n(x,y|\, d(\mu_1 \otimes \mu_2 \otimes \cdots \otimes \mu_r)(x)\, d\nu(y) <M.$$ Now,  by monotone convergence theorem of B. Levi we get that $\lim_{n} s_n(x,y) =s$ exists and finite for almost all $(x,y).$ Also, $s \in L^{1}(\mu_1 \otimes \mu_2 \otimes \cdots \otimes \mu_r \otimes \nu).$ By taking $f=(1,1,\ldots,1)$ and using $|k(x,y|\leq |s(x,y)|,$the proof of Part (i) and Part (ii) follow.
		Since $\{k_n\}_n$ is bounded by $s(x,y)$ we can apply the Lebesgue dominated convergence theorem to get Part (iii).
		To prove the last part, note that $|k_n(x,y)| \leq s(x,y)$ for all $n \in \mathbb{N}$ and almost all $(x,y),$ where $k_n(x,y)= \sum_{j=1}^n h_j(y)\,g_j(x) f(x).$ $s \in L^1(\mu_1 \otimes \mu_2 \otimes \cdots \otimes \mu_r\otimes \nu)$ implies that $s(x, \cdots) \in L^1(\mu_1 \otimes \mu_2 \otimes \cdots \otimes \mu_r)$ for almost all $x \in X_1 \otimes X_2 \otimes \cdots \otimes X_r.$
		Now, the proof of Part (iv) follows using Lebesgue's dominated convergence theorem.
	\end{proof}
	Next theorem present a characterization of the multilinear nuclear operators for finite measure spaces. 
	\begin{theorem} \label{finiteDel} Let $(X_i, \mu_i), 1 \leq i \leq r$ and $(Y, \nu)$ be finite measure spaces. Let $1 \leq p_i,p <\infty, 1 \leq i \leq r$ and let $p_i', q$ be such that $\frac{1}{p_i}+ \frac{1}{p_i'}=1, \frac{1}{p}+\frac{1}{q}=1$ for $1 \leq i \leq r.$ Let $T: L^{p_1}(\mu) \times L^{p_2}(\mu_2) \times \cdots \times L^{p_r}(\mu_r) \rightarrow L^p(\nu)$ be a multilinear operator. Then $T$ is a nuclear operator if and only if there exist sequences $\{g_n\}_n$ with $g_n=(g_{n1}, g_{n2}, \ldots, g_{nr})$ and $\{h_n\}_n$ in $L^{p_1'}(\mu_1) \times L^{p_2'}(\mu_2) \times \cdots \times L^{p_r'}(\mu_r)$ and $L^p(\nu)$ respectively such that $\sum_{n} \|g_n\|_{L^{p_1'}(\mu_1) \times L^{p_2'}(\mu_2) \times \cdots \times L^{p_r'}(\mu_r)} \|h_n\|_{L^p(\nu)}<\infty$ and for all $f = (f_1,f_2,\ldots, f_r) \in L^{p_1}(\mu) \times L^{p_2}(\mu_2) \times \cdots \times L^{p_r}(\mu_r) $ we have
		\begin{align*}
		&(Tf)(y)= \int_{X _1,X_2,\cdots,X_r} \left( \sum_{n=1}^\infty g_n(x)\, h_n(x) \right) f(x)\, d(\mu_1 \otimes \mu_2 \otimes \cdots \otimes \mu_r)(x) \\
		& = \int_{X_1} \int_{X_2} \cdots \int_{X_r} \left( \sum_{n=1}^\infty g_{n1}(x_1)\, g_{n2}(x_2)\ldots g_{nr}(x_r)\, h_n(x) \right) \\
		& \hspace{4cm} \times f_1(x_1) \,f_2(x_2) \ldots f_r(x_r)\, d\mu_1(x_1)\, d\mu_2(x_2) \cdots d\mu_r(x_r)
		\end{align*}
		for almost every $y \in Y.$
	\end{theorem}
	\begin{proof} Let $T:L^{p_1}(\mu) \times L^{p_2}(\mu_2) \times \cdots \times L^{p_r}(\mu_r) \rightarrow L^p(\nu) $ be a nuclear operator. Then, by definition,  there exist $\{g_n\}_n \subset  L^{p_1'}(\mu_1) \times L^{p_2'}(\mu_2) \times \cdots \times L^{p_r'}(\mu_r)$ and $\{h_n\}_n \subset L^p(\nu)$ such that  $$\sum_{n} \|g_n\|_{L^{p_1'}(\mu_1) \times L^{p_2'}(\mu_2) \times \cdots \times L^{p_r'}(\mu_r)} \|h_n\|_{L^p(\nu)}<\infty$$ and 
		$$Tf = \sum_{n} \langle f, g_n \rangle h_n,$$ where $\langle f,g_n \rangle = \langle f_1, g_{n1} \rangle \langle f_2, g_{n2} \rangle \cdots \langle f_r, g_{nr} \rangle$ with $$\langle f_i, g_{ni} \rangle = \int_{X_i} f_i(x_1)\, g_{ni}(x_i)\,d\mu_i(x_i)\,\,\,\,\, 1 \leq i \leq r.$$
		Thus,  
		$$Tf(y)= \sum_{n=1}^\infty \left( \int_{X_1, X_2,\cdots,X_r}  f(x) g_n(x)\, d(\mu_1 \otimes \mu_2 \otimes \cdots  \otimes \mu_r)(x)\right) h_n(y),$$ where sums converge in $L^p(\nu).$ Therefore, there exist subsequences $\{\tilde{g}_n\}_n$ and $\{\tilde{h}_n\}_n$ of $\{g_n\}_n$ and $\{h_n\}_n$ respectively such that 
		$$(Tf)(y)= \sum_{n=1}^\infty \langle f, \tilde{g}_n \rangle \, \tilde{h}_n(y)\,\,\,\, \text{for almost every}\,\, y \in Y.$$
		$\{\tilde{g}_n\}_n$ and $\{\tilde{h}_n\}_n$ are being subsequence of $\{g_n\}_n$ and $\{h_n\}_n$ respectively also satisfy $\sum_{n=1}^\infty \|\tilde{g}_n\|_{L^{p_1'}(\mu_1) \times L^{p_2'}(\mu_2) \times \cdots \times L^{p_r'}(\mu_r)} \|\tilde{h}_n\|_{L^p(\nu)}<\infty.$
		By Lemma \ref{FDel}(iv), we get 
		\begin{align*}
		& (Tf)(y)= \sum_{n=1}^\infty \left( \int_{X_1,X_2,\ldots X_r} \tilde{g}_n(x) f(x) \,d(\mu_1 \otimes \mu_2 \otimes \cdots \otimes \mu_r)(x) \right)\, \tilde{h}_n(y)  \\ 
		& = \lim_{n} \sum_{j=1}^n \left( \int_{X_1,X_2,\ldots X_r} \tilde{g}_j(x) f(x) \,d(\mu_1 \otimes \mu_2 \otimes \cdots \otimes \mu_r)(x) \right)\, \tilde{h}_j(y) \\
		& = \lim_{n} \int_{X_1,X_2,\ldots X_r} \left( \sum_{j=1}^n \tilde{g}_j(x) \, \tilde{h}_j(y)  \right)\, f(x) \,d(\mu_1 \otimes \mu_2 \otimes \cdots \otimes \mu_r)(x) \\
		& = \int_{X_1,X_2,\ldots X_r} \left( \sum_{n=1}^\infty \tilde{g}_n(x) \, \tilde{h}_n(y)  \right)\, f(x) \,d(\mu_1 \otimes \mu_2 \otimes \cdots \otimes \mu_r)(x)\,\,\, a.e. y \in Y..
		\end{align*}
		Conversely, assume that there exist sequences $\{g_n\}_n$ with $g_n=(g_{n1}, g_{n2}, \ldots, g_{nr})$ and $\{h_n\}_n$ in $L^{p_1'}(\mu_1) \times L^{p_2'}(\mu_2) \times \cdots \times L^{p_r'}(\mu_r)$ and $L^p(\nu)$ respectively such that $\sum_{n} \|g_n\|_{L^{p_1'}(\mu_1) \times L^{p_2'}(\mu_2) \times \cdots \times L^{p_r'}(\mu_r)} \|h_n\|_{L^p(\nu)}<\infty$ and for $f \in  L^{p_1}(\mu) \times L^{p_2}(\mu_2) \times \cdots \times L^{p_r}(\mu_r) $ we have
		$$(Tf)(y)= \int_{X_1,X_2,\cdots,X_r} \left( \sum_{n=1}^\infty g_n(x)\, h_n(y) \right) f(x)\, d(\mu_1 \otimes \mu_2 \otimes \cdots \otimes \mu_r)(x)$$
		Again, using Lemma \ref{FDel}(iv), we get 
		\begin{align*}
		& (Tf)(y)= \int_{X_1,X_2,\cdots,X_r} \left( \sum_{n=1}^\infty g_n(x)\, h_n(y) \right) f(x)\, d(\mu_1 \otimes \mu_2 \otimes \cdots \otimes \mu_r)(x) \\
		&= \lim_{n} \int_{X_1,X_2,\cdots,X_r} \left( \sum_{j=1}^n g_j(x)\, h_j(y) \right) f(x)\, d(\mu_1 \otimes \mu_2 \otimes \cdots \otimes \mu_r)(x) \\ 
		& = \lim_{n}\sum_{j=1}^n \left( \int_{X_1,X_2,\cdots,X_r} f(x) g_j(x)  \, d(\mu_1 \otimes \mu_2 \otimes \cdots \otimes \mu_r)(x)  \right) h_j(y) \\
		&= \sum_{n=1}^\infty \left( \int_{X_1,X_2,\cdots,X_r} f(x) g_n(x)  \, d(\mu_1 \otimes \mu_2 \otimes \cdots \otimes \mu_r)(x)  \right) h_n(y) \\
		&= \sum_{n} \langle f, g_n \rangle h_n(y) \,\,\,\, \,\textnormal{almost every}\, y \in Y.
		\end{align*}
		To show that $Tf= \sum_{n} \langle f, g_n \rangle h_n$ in $L^p(\nu).$ We  define $t_n= \sum_{j=1}^n \langle f, g_j\rangle h_j,$ then the sequence $\{t_n\}_n \subset L^p(\nu)$ and for every $n \in \mathbb{N},$
		\begin{align*}
		&|t_n(y)| \leq \|f\|_{L^{p_1}(\mu_1) \times L^{p_2}(\mu_2) \times \cdots \times L^{p_r}(\mu_r)} \sum_{j=1}^n \|g_j\|_{L^{p_1'}(\mu_1) \times L^{p_2'}(\mu_2) \times L^{p_r'}(\mu_r)} |h_j(y)| \\
		& \leq \|f\|_{L^{p_1}(\mu_1) \times L^{p_2}(\mu_2) \times \cdots \times L^{p_r}(\mu_r)} \sum_{j=1}^\infty \|g_j\|_{L^{p_1'}(\mu_1) \times L^{p_2'}(\mu_2) \times L^{p_r'}(\mu_r)} |h_j(y)| = q(y)\, (\text{say}).
		\end{align*}
		Then, $q(y)$ is well-defined and $q \in L^p(\nu).$ Infact,  $q$ is the limit of an increasing sequence $\{q_n\}_n $ in $L^p(\nu),$ where $$q_n(y)= \|f\|_{L^{p_1}(\mu_1) \times L^{p_2}(\mu_2) \times \cdots \times L^{p_r}(\mu_r)} \sum_{j=1}^n \|g_j\|_{L^{p_1'}(\mu_1) \times L^{p_2'}(\mu_2) \times L^{p_r'}(\mu_r)} |h_j(y)|,$$ and  
		\begin{align*}
		& \|q_n\|_{L^p(\nu)} \leq  \|f\|_{L^{p_1}(\mu_1) \times L^{p_2}(\mu_2) \times \cdots \times L^{p_r}(\mu_r)} \\
		& \hspace{1cm} \times \sum_{j=1}^n \|g_j\|_{L^{p_1'}(\mu_1) \times L^{p_2'}(\mu_2) \times L^{p_r'}(\mu_r)} \|h_j(y)\|_{L^p(\nu)} <M.
		\end{align*}
		So, by applying monotone convergence theorem of B. Levi we deduce that $q \in L^p(\nu).$ Finally, Lebesgue dominated convergence theorem gives that we get that $t_n \rightarrow Tf$ in $L^p(\nu).$ 
	\end{proof}
	
	Now, we are  ready to extend above theorem from finite measure spaces to $\sigma$-finite measure spaces. First, we prove the following lemma.
	
	\begin{lemma} \label{Dsigma} Let $(X_i, \mu_i), 1 \leq i \leq r$ and $(Y, \nu)$ be $\sigma$-finite measure spaces. Let $1 \leq p_i,p <\infty, 1 \leq i \leq r$ and let $p_i', q$ be such that $\frac{1}{p_i}+ \frac{1}{p_i'}=1, \frac{1}{p}+\frac{1}{q}=1$ for $1 \leq i \leq r.$ Suppose $f \in L^{p_1}(\mu_1) \times L^{p_2}(\mu_2) \times \cdots \times L^{p_r}(\mu_r),$ where $f= (f_1,f_2,\ldots, f_r)$ such that each $f_i \in L^{p_i}(\mu_i),$ and $\{g_n\}_n$ with $g_n=(g_{n1}, g_{n2}, \ldots, g_{nr})$ and $\{h\}_n$ are two sequences in $L^{p_1'}(\mu_1) \times L^{p_2'}(\mu_2) \times \cdots \times L^{p_r'}(\mu_r)$ and $L^p(\nu)$ respectively such that $\sum_{n} \|g_n\|_{L^{p_1'}(\mu_1) \times L^{p_2'}(\mu_2) \times \cdots \times L^{p_r'}(\mu_r)} \|h_n\|_{L^p(\nu)}<\infty.$ Then the parts (i) and (iv) of Lemma \ref{FDel} hold.   
	\end{lemma}
	\begin{proof} Since $(X_i, \mu_i), 1 \leq i \leq r$ and $(Y,\nu)$ are $\sigma$-finite measure spaces. Then there exist sequences $\{X_i^{k_i}\}_{k_i}$  and $\{Y^k\}_k$ such that $\mu_i(X_i^{k_i})< \infty,$ $\nu(Y^k)<\infty$ and $X_i= \bigcup_{k_i} X_i^{k_i}, \,Y= \bigcup_{k} Y^k.$  Now, consider the finite measure spaces $(X_i^{k_i}, \mu_i^{k_i})$ and $(Y^k, \nu^k),$ where $\mu_i^{k_i},\, 1 \leq i \leq r$ and $\nu_k$ are the restriction of $\mu_i, \, 1 \leq i \leq r$ and $\nu$ on $X_i^{k_i}, 1 \leq i \leq r$ and $Y^k$ respectively. Now, restrict the function $g_n$ on $X_1^{k_1}\times X_2^{k_2} \times \cdots X_r^{k_r}$ and $h_n$ on $Y^k.$ Then for each $k_i$ and $k,$  we have 
		$$\sum_{n=1}^\infty \|g\|_{L^{p_1'}(\mu_1^{k_1})\times L^{p_2'}(\mu_2^{k_2}) \times \cdots \times L^{p_r'}(\mu_r^{k_r})} \|h_n\|_{L^p(\nu^k)}<\infty. $$
		By Lemma \ref{FDel}(a) we know that $\sum_{j=1}^\infty g_j(x)\, h_j(y)$ converge absolutely for almost every $(x,y) \in  \times X_1^{k_1}\times X_2^{k_2} \times \cdots X_r^{k_r} \times Y^k.$ Therefore, $\sum_{j=1}^\infty g_j(x)\, h_j(y)$ converge absolutely for almost every $(x,y)= (x_1,x_2, \ldots,x_r,y)  \in X_1 \times X_2 \times  \cdots \times X_r \times Y.$ 
		
		We know, from the part (i) that the series $\sum_{j=1}^\infty g_j(x)\, h_j(y) f(x)$ converge absolutely for almost every $(x,y)$. The part (iv) follows using Lebesgue dominated convergence theorem as in the proof of converse part of Theorem \ref{finiteDel}.
	\end{proof} 
	Here is our main theorem of this section on the characterization of multilinear nuclear operators on $\sigma$-finite measure spaces.
	\begin{theorem}\label{sigmafinitesVDelgado}
		Let $(X_i, \mu_i), 1 \leq i \leq r$ and $(Y, \nu)$ be $\sigma$-finite measure spaces. Let $1 \leq p_i,p <\infty, 1 \leq i \leq r$ and let $p_i', q$ be such that $\frac{1}{p_i}+ \frac{1}{p_i'}=1, \frac{1}{p}+\frac{1}{q}=1$ for $1 \leq i \leq r.$ Let $T: L^{p_1}(\mu) \times L^{p_2}(\mu_2) \times \cdots \times L^{p_r}(\mu_r) \rightarrow L^p(\nu)$ be a multilinear operator. Then $T$ is a nuclear operator if and only if there exist sequences $\{g_n\}_n$ with $g_n=(g_{n1}, g_{n2}, \ldots, g_{nr})$ and $\{h_n\}_n$ in $L^{p_1'}(\mu_1) \times L^{p_2'}(\mu_2) \times \cdots \times L^{p_r'}(\mu_r)$ and $L^p(\nu)$ respectively such that $\sum_{n} \|g_n\|_{L^{p_1'}(\mu_1) \times L^{p_2'}(\mu_2) \times \cdots \times L^{p_r'}(\mu_r)} \|h_n\|_{L^p(\nu)}<\infty$ and for all $f = (f_1,f_2,\ldots, f_r) \in L^{p_1}(\mu) \times L^{p_2}(\mu_2) \times \cdots \times L^{p_r}(\mu_r) $ we have
		\begin{align*}
		(Tf)(y)&= \int_{X_1,X_2,\cdots,X_r} \left( \sum_{n=1}^\infty g_n(x)\, h_n(y) \right) f(x)\, d(\mu_1 \otimes \mu_2 \otimes \cdots \otimes \mu_r)(x) \\
		& = \int_{X_1} \int_{X_2} \cdots \int_{X_r} \left( \sum_{n=1}^\infty g_{n1}(x_1)\, g_{n2}(x_2)\ldots g_{nr}(x_r)\, h_n(y) \right) \\
		& \hspace{4cm} \times f_1(x_1) \,f_2(x_2) \ldots f_r(x_r)\, d\mu_1(x_1)\, d\mu_2(x_2) \cdots d\mu_r(x_r)
		\end{align*}
		for almost every $y \in Y.$
	\end{theorem}
	\begin{proof} The proof of theorem follows the same line as the proof of Theorem \ref{finiteDel} by replacing Lemma \ref{FDel}(iv) by Lemma \ref{Dsigma}.
	\end{proof} 
	\begin{remark} \label{VDelgado}
		An analogue of the characterization above theorem holds for $s$- nuclear operators, $0<s \leq 1,$ replacing the terms $\|g_n\|_{L^{p_1'}(\mu_1) \times L^{p_2'}(\mu_2) \times \cdots \times L^{p_r'}(\mu_r)} \|h_n\|_{L^p(\nu)}$ by $\|g_n\|_{L^{p_1'}(\mu_1) \times L^{p_2'}(\mu_2) \times \cdots \times L^{p_r'}(\mu_r)}^s \|h_n\|_{L^p(\nu)}^s$ in the sum, i.e., under the condition that 
		$$\sum_{n} \|g_n\|_{L^{p_1'}(\mu_1) \times L^{p_2'}(\mu_2) \times \cdots \times L^{p_r'}(\mu_r)}^s \|h_n\|_{L^p(\nu)}^s<\infty.$$
	\end{remark}

	\section{Discrete and periodic $s$-nuclear, $0<s \leq 1,$  multilinear pseudo-differential operators on Lebesgue spaces}
	
	This section is devoted to study the $L^p$-nuclearity of multilinear pseudo-differential operator defined on $\mathbb{Z}^n$ and the torus $\mathbb{T}^n$ with the help of multilinear version of Delgado's theorem for $s$-nuclear, $0<s \leq 1,$ operator, namely, Theorem \ref{sigmafinitesVDelgado} and Remark \ref{VDelgado} in the previous section.   
	
	\subsection{$s$-Nuclearity of multilinear pseudo-differential operators on $\mathbb{Z}^n$}  
	In this subsection we study multilinear discrete pseudo-differential operators. We present several results concerning to the $s$-nuclearity , $0<s \leq 1,$ of multilinear discrete pseudo-differential operator. We begin with the following characterization.
	\begin{theorem} \label{vdiscrete}
		Let $a$ be a measurable function defined on $\mathbb{Z}^n \times (\mathbb{T}^{nr}).$ The multilinear pseudo-differential operator $T_a: L^{p_1}(\mathbb{Z}^n)\times L^{p_2}(\mathbb{Z}^n) \times \cdots L^{p_r}(\mathbb{Z}^n) \rightarrow L^{p}(\mathbb{Z}^n), 1 \leq p_i< \infty,$ for all $1 \leq i \leq r$ is a $s$-nuclear, $0 < s \leq 1, $ operator if and only if  the following decomposition holds:
		$$a(x,\xi)=e^{-i2\pi \tilde{x}\cdot \xi}\sum_{k}   h_k(x)\mathscr{F}_{\mathbb{Z}^{nr}}(g_k)(-\xi),\xi\in\mathbb{T}^{nr},x\in\mathbb{Z}^{n},$$ 
		where $\tilde{x}= (x,x, \ldots,x)  \in (\mathbb{Z}^n)^r;$ $\{h_k\}_k$ and $\{g_k\}_k$ with $g_k= (g_{k1}, g_{k2}, \ldots, g_{kr})$  are two sequences in $L^{p}(\mathbb{Z}^n)$ and $L^{p_1'}(\mathbb{Z}^n)\times L^{p_2'}(\mathbb{Z}^n) \times \cdots \times L^{p_r'}(\mathbb{Z}^n)$ respectively such that $\sum_{n=1}^\infty \|h_n\|_{L^{p}(\mathbb{Z}^n)}^s \|g_n\|_{L^{p_1'}(\mathbb{Z}^n)\times L^{p_2'}(\mathbb{Z}^n) \times \cdots \times L^{p_r'}(\mathbb{Z}^n)}^s <\infty.$
	\end{theorem}
	\begin{proof}
		Let $T_a$ be a $s$-nuclear, $0 < s \leq 1, $ operator. Therefore, by Remark \ref{VDelgado} there exist sequences $\{g_k\}_k \in L^{p_1'}(\mathbb{Z}^n)\times L^{p_2'}(\mathbb{Z}^n) \times \cdots \times L^{p_r'}(\mathbb{Z}^n)$ and $\{h_k\}_k \in L^{p}(\mathbb{Z}^n)$ with $\sum_{n} \|h_k\|_{L^{p}(\mathbb{Z}^n)}^s \|g_k\|_{L^{p_1'}(\mathbb{Z}^n)\times L^{p_2'}(\mathbb{Z}^n) \times \cdots \times L^{p_r'}(\mathbb{Z}^n)}^s <\infty$ such that for all $f \in L^{p_1}(\mathbb{Z}^n)\times L^{p_2}(\mathbb{Z}^n) \times \cdots L^{p_r}(\mathbb{Z}^n),$ we have 
		
		\begin{equation}
		(T_af)(x)= \sum_{y \in \mathbb{Z}^{nr}} \sum_{k} h_k(x) \, g_k(y) \, f(y),
		\end{equation}
		where $g_k=(g_{k_1},g_{k_2},\ldots g_{k_r}).$ \\
		On the other hand, set $\xi= (\xi_1,\xi_2,\ldots,\xi_r) \in \mathbb{T}^{nr},$ we have  
		\begin{align*}
		(T_af)(x) &= \int_{\mathbb{T}^{nr}} a(x, \xi)\,  e^{2\pi i x \cdot (\xi_1+\xi_2+\cdots+\xi_r)} \mathscr{F}_{\mathbb{Z}^n}f_1(\xi_1)\,\mathscr{F}_{\mathbb{Z}^n}f_2(\xi_2), \ldots, \mathscr{F}_{\mathbb{Z}^n}f_r(\xi_r) \, d\xi \\
		& =  \int_{\mathbb{T}^{nr}} a(x, \xi)\,e^{2\pi i x \cdot (\xi_1+\xi_2+\cdots+\xi_r)}\, \sum_{y \in \mathbb{Z}^{nr}}f_1(y_1)\,f_2(y_2),\ldots f_r(y_r)\, e^{-2 \pi i y\cdot \xi} \, d\xi \\
		& = \int_{\mathbb{T}^{nr}} a(x,\xi) e^{2 \pi i(\tilde{x}-y)\cdot \xi}\,\sum_{y \in \mathbb{Z}^{nr}}f_1(y_1)\,f_2(y_2),\ldots f_r(y_r)\, d\xi.
		\end{align*}
		
		Therefore, for all $f=(f_1,f_2, \ldots, f_r) \in L^{p_1}(\mathbb{Z}^n)\times L^{p_2}(\mathbb{Z}^n) \times \cdots \times L^{p_r}(\mathbb{Z}^n),$ we get
		\begin{equation}\label{vpe}
		\int_{\mathbb{T}^{nr}} a(x,\xi) e^{2 \pi i(\tilde{x}-y)\cdot \xi}\,\sum_{y \in \mathbb{Z}^{nr}}f(y)\, d\xi= \sum_{y \in \mathbb{Z}^{nr}} \sum_{k} h_k(x) \, g_k(y) \, f(y).
		\end{equation} 
		Since \eqref{vpe} holds for all $f=(f_1,f_2, \ldots, f_r) \in L^{p_1}(\mathbb{Z}^n)\times L^{p_2}(\mathbb{Z}^n) \times \cdots\times L^{p_r}(\mathbb{Z}^n).$ For any $\ell= (\ell_1, \ell_2, \ldots, \ell_r) \in (\mathbb{Z}^{n})^r,$ by choose $f= f_\ell$ such that 
		$$f_\ell(y)= \begin{cases} 1 & \text{if} \,\ell_j=y_j\,\, \forall\, j \in \{1,2, \ldots, r\}, \\ 0 & \text{otherwise}.  \end{cases}$$ Equation \eqref{vpe} in turn gives 
		$$\mathscr{F}^{-1}_{\mathbb{Z}^{nr}}[a(x,\cdot)](\tilde{x}-\ell):= \int_{\mathbb{T}^{nr}} a(x,\xi) e^{2 \pi i(\tilde{x}-\ell)\cdot \xi}\, d\xi=  \sum_{k} h_k(x) \, g_k(\ell),$$ where $\tilde{x}= (x,x, \ldots,x) \in (\mathbb{Z}^n)^r.$ So, by the Fourier inversion formula for the discrete Fourier transform we have
		\begin{align*}
		a(x,\xi)&= \mathscr{F}_{\mathbb{Z}^{nr}}[\mathscr{F}_{\mathbb{Z}^{nr}}^{-1}(a(x,\cdot))](\xi) = \sum_{\ell \in \mathbb{Z}^{nr} }e^{-i2\pi \xi\cdot \ell}\mathscr{F}_{\mathbb{Z}^{nr}}^{-1}(a(x,\cdot))(\ell)\\
		&=\sum_{\ell \in \mathbb{Z}^{nr} }e^{-i2\pi \xi\cdot (\tilde{x}-\ell)}\mathscr{F}_{\mathbb{Z}^{nr}}^{-1}(a(x,\cdot))(\tilde{x}-\ell)\\
		&=\sum_{\ell \in \mathbb{Z}^{nr} }e^{-i2\pi \xi\cdot (\tilde{x}-\ell)}\sum_{k}h_k(x)g_k(\ell)\\
		&=e^{-i2\pi \xi\cdot \tilde{x}}\sum_{\ell \in \mathbb{Z}^{nr} }\sum_{k}h_k(x)e^{i2\pi\xi \ell}g_k(\ell)\\
		&=e^{-i2\pi \xi\cdot \tilde{x}}\sum_{k}   h_k(x)\sum_{\ell \in \mathbb{Z}^{nr} }e^{i2\pi\xi \ell}g_k(\ell)\\
		&=e^{-i2\pi \xi\cdot \tilde{x}}\sum_{k}   h_k(x)\mathscr{F}_{\mathbb{Z}^{nr}}(g_k)(-\xi).
		\end{align*}
		
		Conversely, assume that there exist $\{h_k\}_k$ and $\{g_k\}_k$ in $L^{p}(\mathbb{Z}^n)$ and $L^{p_1'}(\mathbb{Z}^n)\times L^{p_2'}(\mathbb{Z}^n) \times \cdots\times  L^{p_r'}(\mathbb{Z}^n)$ respectively with  $$\sum_{k=1}^\infty \|h_k\|_{L^p(\mathbb{Z}^n)}^s \|g_k\|_{L^{p_1'}(\mathbb{Z}^n)\times L^{p_2'}(\mathbb{Z}^n) \times \cdots L^{p_r'}(\mathbb{Z}^n)}^s <\infty$$ such that 
		$$ a(x,\xi)=e^{-i2\pi \xi\cdot \tilde{x}}\sum_{k}   h_k(x)\mathscr{F}_{\mathbb{Z}^{nr}}(g_k)(-\xi),\xi\in\mathbb{T}^{nr},x\in\mathbb{Z}^{n}.$$ Then, we have the identity
		
		$$\int_{   \mathbb{T}^{nr}} a(x,\xi) e^{2 \pi i(\tilde{x}-\ell)\cdot \xi}\, d\xi=  \sum_{k} h_k(x) \, g_k(\ell)$$  for all $\ell=(\ell_1,\ell_2, \ldots, \ell_r) \in (\mathbb{Z}^{n})^r$. Using this we have 
		\begin{align*}
		(T_af)(x)&= \int_{\mathbb{T}^{nr}}  a(x, \xi)\,  e^{2\pi i x \cdot (\xi_1+\xi_2+\cdots+\xi_r)} \mathscr{F}_{\mathbb{Z}^n}f_1(\xi_1)\,\mathscr{F}_{\mathbb{Z}^n}f_2(\xi_2), \ldots, \mathscr{F}_{\mathbb{Z}^n}f_r(\xi_r) \, d\xi  \\
		& = \sum_{y \in \mathbb{Z}^{nr}}\, \int_{\mathbb{T}^{nr}}\,a(x,\xi) e^{2 \pi i(\tilde{x}-y)\cdot \xi} f(y)\, d\xi \\
		& = \sum_{y \in \mathbb{Z}^{nr}} \left( \sum_{k} h_k(x) \, g_k(y) \right) f(y) \,\,\, \text{almost every}\, x \in \mathbb{Z}^n.
		\end{align*}
		Hence, by Remark \ref{VDelgado}, $T_a$ is a $s$-nuclear, $0 < s \leq 1, $ operator.
	\end{proof}
	
	The following theorem present a necessary condition on symbols for the associated multilinear discrete pseudo-differential operator to be a nuclear operator.

	\begin{theorem}
		Let $a$ be a measurable function on $\mathbb{Z}^n \times \mathbb{T}^{nr}$ such that the miltilinear pseudo-differential operator $T_a:L^{p_1}(\mathbb{Z}^n)\times L^{p_2}(\mathbb{Z}^n) \times \cdots L^{p_r}(\mathbb{Z}^n) \rightarrow L^{p}(\mathbb{Z}^n)$ is a nuclear operator. Then, the sequence 
		$$\{\|  \int_{\mathbb{T}^{nr}} a(x,\xi) e^{2 \pi i(\tilde{x}-\ell)\cdot \xi}\, d\xi    \|_{L^{p}(\mathbb{Z}^n_x)}\}_{(\ell_1,\ell_2, \ldots,\ell_r) \in (\mathbb{Z}^n)^r} $$ is in  $ L^{p_1'}(\mathbb{Z}^n)\times L^{p_2'}(\mathbb{Z}^n) \times \cdots L^{p_r'}(\mathbb{Z}^n).$
	\end{theorem}
	\begin{proof}
		Since $T_a$ is nuclear operator. So, by Theorem \ref{vdiscrete}, 
		there exist $\{h_k\}_k$ and $\{g_k\}_k$ in $L^{p}(\mathbb{Z}^n)$ and $L^{p_1'}(\mathbb{Z}^n)\times L^{p_2'}(\mathbb{Z}^n) \times \cdots L^{p_r'}(\mathbb{Z}^n)$ respectively with  $$\sum_{n=1}^\infty \|h_n\|_{L^{p}(\mathbb{Z}^n)} \|g_n\|_{L^{p_1'}(\mathbb{Z}^n)\times L^{p_2'}(\mathbb{Z}^n) \times \cdots L^{p_r'}(\mathbb{Z}^n)} <\infty$$ such that  $$\int_{\mathbb{T}^{nr}} a(x,\xi) e^{2 \pi i(\tilde{x}-\ell)\cdot \xi}\, d\xi=  \sum_{k} h_k(x) \, g_k(\ell)$$ holds for all $\ell=(\ell_1,\ell_2, \ldots, \ell_r) \in (\mathbb{Z}^{n})^r$. Which is equivalent to saying that 
		$$(\mathscr{F}_{\mathbb{T}^{nr}}a)(x, x-\ell_1, x-\ell_2, \ldots, x-\ell_r) =  \sum_{k} h_k(x) \, g_k(\ell)$$ holds for all $\ell=(\ell_1,\ell_2, \ldots, \ell_r) \in (\mathbb{Z}^{n})^r$. 
		Now, By Minkowski's inequality, we get 
		\begin{align*}
		&F(\ell):= \left( \sum_{x \in \mathbb{Z}^n} |(\mathscr{F}_{\mathbb{T}^{nr}}a)(x, x-\ell_1, x-\ell_2, \ldots, x-\ell_r)|^{p} \right)^{\frac{1}{p}} \\ 
		&  = \left( \sum_{x \in \mathbb{Z}^n} \left| \sum_{k} h_k(x) \, g_k(\ell)\right|^{p} \right)^{\frac{1}{p}} \\
		&\leq \sum_{k} \left( \sum_{x \in \mathbb{Z}^n} |h_k(x) \, g_k(\ell)|^{p} \right)^{\frac{1}{p}} \leq \sum_{k} \|h_k\|_{L^{p}(\mathbb{Z})}|g_k(\ell)|,
		\end{align*}
		and thus we have 
		\begin{align*}
		&\Vert F(\ell)\Vert_{L^{p_1'}(\mathbb{Z}^n) \times \cdots \times  L^{p_r'}(\mathbb{Z}^n) }\leq \Vert \sum_{k} \|h_k\|_{L^{p}(\mathbb{Z})}|g_k(\ell)|\Vert_{L^{p_1'}(\mathbb{Z}^n) \times \cdots \times  L^{p_r'}(\mathbb{Z}^n) } \\
		&\leq \sum_{k} \|h_k\|_{L^{p}(\mathbb{Z})}\Vert g_k(\ell)\Vert_{L^{p_1'}(\mathbb{Z}^n) \times \cdots \times  L^{p_r'}(\mathbb{Z}^n) }\\
		&= \sum_{k} \|h_k\|_{L^{p}(\mathbb{Z})}\prod_{j=1}^r\Vert g_{kj} \Vert_{L^{p_j'}(\mathbb{Z}^{n})}< \infty.
		\end{align*} 
		So, we finish the proof.
	\end{proof}
	
	\begin{theorem}
		Let $a$ be a measurable function on $\mathbb{Z}^n \times \mathbb{T}^{nr}$ such that the multilinear pseudo-differential operator $T_a: L^{p_1}(\mathbb{Z}^n) \times \cdots \times  L^{p_r}(\mathbb{Z}^n) \rightarrow L^{p}(\mathbb{Z}^n)$ is a nuclear operator. Then, we have 
		$$ \left\{ \left\| \int_{\mathbb{T}^{nr}} a(x,\xi) e^{2 \pi i(\tilde{x}-\ell)\cdot \xi}\, d\xi \right\|_{L^{p_1'}(\mathbb{Z}^n) \times \cdots \times L^{p_r'}(\mathbb{Z}^n)} \right\}_{x \in (\mathbb{Z}^n)}  \in L^{p}(\mathbb{Z}^n),$$ where $p_i'$ is such that $\frac{1}{p_i}+\frac{1}{p_i'}=1 $ for every $ 1 \leq i \leq r.$
	\end{theorem}
	\begin{proof}
		Since $T_a$ is nuclear operator. So, by Theorem \ref{vdiscrete}, 
		there exist $\{h_k\}_k$ and $\{g_k\}_k$ in $L^{p}(\mathbb{Z}^n)$ and $L^{p_1'}(\mathbb{Z}^n)\times \cdots \times L^{p_r'}(\mathbb{Z}^n)$ respectively with  $$ \sum_{k} \|h_k\|_{L^{p}(\mathbb{Z}^n)} \|g_k\|_{L^{p_1'}(\mathbb{Z}^n)\times \cdots \times L^{p_r'}(\mathbb{Z}^n)} <\infty$$ such that  $$\int_{\mathbb{T}^{nr}} a(x,\xi) e^{2 \pi i(\tilde{x}-\ell)\cdot \xi}\, d\xi=  \sum_{k} h_k(x) \, g_k(\ell)$$ holds for all $\ell=(\ell_1,\ell_2, \ldots, \ell_r) \in (\mathbb{Z}^{n})^r$.  Then 
		\begin{align*}
		&F(x)= \left\| \int_{T^{nr}} a(x,\xi) e^{2 \pi i(\tilde{x}-\ell)\cdot \xi}\, d\xi \right\|_{L^{p_1'}(\mathbb{Z}^n) \times L^{p_2'}(\mathbb{Z}^n) \times \cdots \times L^{p_r'}(\mathbb{Z}^n)} \\ 
		& = \left\| \sum_{k} h_k(x) \, g_k(\ell) \right\|_{L^{p_1'}(\mathbb{Z}^n) \times L^{p_2'}(\mathbb{Z}^n) \times \cdots \times L^{p_r'}(\mathbb{Z}^n)}\\
		& \leq   \sum_{k} \left\|  h_k(x) \, g_k(\ell) \right\|_{L^{p_1'}(\mathbb{Z}^n) \times L^{p_2'}(\mathbb{Z}^n) \times \cdots \times L^{p_r'}(\mathbb{Z}^n)} \\
		& \leq \sum_{k} |h_k(x)|\|g_k\|_{L^{p_1'}(\mathbb{Z}^n) \times L^{p_2'}(\mathbb{Z}^n)}= \sum_{k} |h_k(x)|\prod_{j=1}^r \|g_{kj}\|_{L^{p_j'}(\mathbb{Z}^n)}
		\end{align*}
		and therefore, by Minkowski's inequality, we get 
		\begin{align*}
		&     \left( \sum_{ x \in \mathbb{Z}^n} |F(x)|^p \right)^{\frac{1}{p}} = \left( \sum_{x \in \mathbb{Z}^n} \left( \sum_{k} |h_k(x)|\prod_{j=1}^r \|g_{kj}\|_{L^{p_j'}(\mathbb{Z}^n)}  \right)^{p} \right)^{\frac{1}{p}} \\
		&\leq \sum_{k} \left( \sum_{x \in \mathbb{Z}^n} |h_k(x)|^{p} \right)^{\frac{1}{p}} \prod_{j=1}^r \|g_{kj}\|_{L^{p_j'}(\mathbb{Z}^n)} \leq \sum_{k} \|h_k\|_{L^{p}(\mathbb{Z}^n)} \prod_{j=1}^r \|g_{kj}\|_{L^{p_j'}(\mathbb{Z}^n)} < \infty. 
		\end{align*}
		Hence, the conclusion follows.
	\end{proof}

	\subsection{$s$-Nuclearity of  multilinear pseudo-differential operators on the torus $\mathbb{T}^n$} In this section, we study the $s$-nuclearity, $0<s \leq 1,$ of multilinear periodic pseudodifferential operators. We give the following characterization of a multilinear $s$-nuclear, $0<s \leq 1,$ pseudo-differential operator on $\mathbb{T}^n.$ 
	
	\begin{theorem} Let $m$ be a measurable function on $\mathbb{T}^n \times \mathbb{Z}^{nr}.$ Then the mutlilinear pseudo-differential operator $T_m:L^{p_1}(\mathbb{T}^n) \times \cdots \times L^{p_r}(\mathbb{T}^n) \rightarrow L^p(\mathbb{T}^n),$ $1 \leq p_i,p < \infty $ for $1\leq i \leq r,$ is a $s$-nuclear, $0 < s \leq 1, $ operator if and only if there exist two sequences $\{g_k\}_k$ with $g_{k}= \left(g_{k1},g_{k2}, \ldots, g_{kr}\right)$ and $\{h_k\}_k$ in $L^{p_1'}(\mathbb{T}^n) \times \cdots \times L^{p_r'}(\mathbb{T}^n),\, \frac{1}{p_i}+\frac{1}{p_i'}=1$ for $1\leq i \leq r$ and $L^p(\mathbb{T}^n)$  respectively such that  $\sum_{k} \|g_k\|_{L^{p_1'}(\mathbb{T}^n) \times \cdots \times L^{p_r'}(\mathbb{T}^n)}^s \|h_k\|_{L^p(\mathbb{T})}^s <\infty $ and $$ m(x, \eta)= e^{-i2 \pi \tilde{x} \cdot \eta} \sum_{k} h_k(x)\, (\mathscr{F}_{\mathbb{T}^{nr}}g_k)(-\eta),\,\, \eta \in \mathbb{Z}^{nr}$$ where $\tilde{x}= (x,x,\ldots,x) \in \mathbb{T}^{nr}.$
	\end{theorem}
	\begin{proof} Let $T_m$ is a multilinear $s$-nuclear, $0 < s \leq 1, $ pseudo-differential operator. Then, by Remark \ref{VDelgado}, there exist two sequences $\{g_k\}_k$ with $g_{k}= \left(g_{k1},g_{k2}, \ldots, g_{kr}\right)$ and $\{h_k\}_k$ in $L^{p_1'}(\mathbb{T}^n) \times \cdots \times L^{p_r'}(\mathbb{T}^n)$ and $L^p(\mathbb{T}^n)$ such that  $$\sum_{k} \|g_k\|_{L^{p_1'}(\mathbb{T}^n) \times \cdots \times L^{p_r'}(\mathbb{T}^n)}^s\|h_k\|_{L^p(\mathbb{T})}^s <\infty $$ and for all $f \in  L^{p_1}(\mathbb{T}^n) \times \cdots \times L^{p_r}(\mathbb{T}^n),$ we have 
		
		\begin{equation} \label{pdoto}
		T_mf(x)= \int_{\mathbb{T}^{nr}} \left( \sum_k h_k(x)\, g_k(y) \right)\, f(y) \,dy,
		\end{equation} 
		where $y=\left(y_1, y_2, \cdots, y_r \right) \in T^{nr}.$ 
		
		Now, for any $\eta= (\eta_1, \eta_2, \cdots, \eta_r) \in \mathbb{Z}^{nr},$ 
		define a function $$f_\eta(y)= \left(f_{\eta_i}, f_{\eta_2}, \cdots, f_{\eta_r} \right) \in L^{p_1}(\mathbb{T}^n) \times \cdots \times L^{p_r}(\mathbb{T}^n) $$ such that $f_{\eta_i}(y_i)= e^{i2\pi \eta_i \cdot y_i}$ for  $1 \leq i \leq r.$ Note that 
		\begin{equation*}
		(\mathscr{F}_{\mathbb{T}^n}{f}_{\eta_i})(\xi_i)=  \begin{cases} 1 & \eta_i=\xi_i \\ 0 & \eta_i \neq \xi_i \,\,\,\,\, 1 \leq i \leq r.
		\end{cases}
		\end{equation*}

		Therefore, by definition of periodic pseudo-differential operator and \eqref{pdoto}, we get 
		\begin{align*}
		T_m(f)(x)&=\sum_{\xi\in \mathbb{Z}^{nr}}e^{i2\pi x\cdot (\xi_1+\xi_2+\cdots +\xi_r)  }m(x,\xi)(\mathscr{F}_{\mathbb{T}^n}{f}_{1})(\xi_1)\cdots (\mathscr{F}_{\mathbb{T}^n}{f}_{r})(\xi_r) \\
		& = \int_{\mathbb{T}^{nr}} \left( \sum_k h_k(x)\, g_k(y) \right)\, f(y) \,dy
		\end{align*}
		holds for every $f=(f_1,f_2, \cdots, f_r) \in L^{p_1}(\mathbb{T}^n) \times \cdots \times L^{p_r}(\mathbb{T}^n).$ In particular for  $f= f_\eta,$  we get
		
		\begin{align*}
		e^{i2\pi x\cdot (\eta_1+\eta_2+\cdots +\eta_r)  }\, m(x,\eta) &= \int_{\mathbb{T}^{nr}} \left( \sum_k h_k(x)\, g_k(y) \right) e^{i2 \pi(\eta_1y_1+\eta_2 y_2+\cdots+ \eta_r y_r)} dy \\
		&= \sum_{k} h_k(x) \int_{\mathbb{T}^{nr}} g_k(y)\, e^{i2 \pi \eta \cdot y} \,dy \\ 
		&= \sum_{k} h_k(x)\, (\mathscr{F}_{\mathbb{T}^{nr}}g_k)(-\eta).
		\end{align*}
		Therefore, 
		$$ m(x, \eta)= e^{-i2 \pi \tilde{x} \cdot \eta} \sum_{k} h_k(x)\, (\mathscr{F}_{\mathbb{T}^{nr}}g_k)(-\eta),$$ where $\tilde{x}= (x,x,\ldots,x) \in \mathbb{T}^{nr}.$
		
		Conversely,  assume that there exist two sequences $\{g_k\}_k$ with $g_{k}= \left(g_{k1},g_{k2}, \ldots, g_{kr}\right)$ and $\{h_k\}_k$ in $L^{p_1'}(\mathbb{T}^n) \times \cdots \times L^{p_r'}(\mathbb{T}^n)$ and $L^p(\mathbb{T}^n)$ such that  $$\sum_{k} \|g_k\|_{L^{p_1'}(\mathbb{T}^n) \times \cdots \times L^{p_r'}(\mathbb{T}^n)}^s \|h_k\|_{L^p(\mathbb{T})}^s <\infty $$ and $$ m(x, \eta)= e^{-i2 \pi \tilde{x} \cdot \eta} \sum_{k} h_k(x)\, (\mathscr{F}_{\mathbb{T}^{nr}}g_k)(-\eta),$$ where $\tilde{x}= (x,x,\ldots,x) \in \mathbb{T}^{nr}.$
		
		Therefore, for any $f= (f_1,f_2, \ldots,f_r) \in L^{p_1}(\mathbb{T}^n) \times \cdots \times L^{p_r}(\mathbb{T}^n)$
		\begin{align*}
		T_m(f)(x)& =\sum_{\xi\in \mathbb{Z}^{nr}}e^{i2\pi x\cdot (\xi_1+\xi_2+\cdots +\xi_r)  }m(x,\xi)(\mathscr{F}_{\mathbb{T}^n}{f}_{1})(\xi_1)\cdots (\mathscr{F}_{\mathbb{T}^n}{f}_{r})(\xi_r) \\
		& = \sum_{\xi\in \mathbb{Z}^{nr}} \left( \sum_{k} h_k(x)\, (\mathscr{F}_{\mathbb{T}^{nr}}g_k)(-\xi) \right)\mathscr{F}_{\mathbb{T}^n}{f}_{1})(\xi_1)\cdots (\mathscr{F}_{\mathbb{T}^n}{f}_{r})(\xi_r) \\
		&=  \sum_{\xi\in \mathbb{Z}^{nr}} \left( \sum_{k} h_k(x)\,\int_{\mathbb{T}^{nr}} g_k(y)\, e^{i2\pi \xi \cdot y} \,dy \right)\mathscr{F}_{\mathbb{T}^n}{f}_{1})(\xi_1)\cdots (\mathscr{F}_{\mathbb{T}^n}{f}_{r})(\xi_r) \\ 
		&= \int_{\mathbb{T}^{nr}} \left( \sum_{k} h_k(x)\,g_k(y)\, \right) \left( \sum_{\xi\in \mathbb{Z}^{nr}} e^{i2\pi \xi \cdot y} \mathscr{F}_{\mathbb{T}^n}{f}_{1})(\xi_1)\cdots (\mathscr{F}_{\mathbb{T}^n}{f}_{r})(\xi_r) \right)dy \\
		&= \int_{\mathbb{T}^{nr}} \left( \sum_{k} h_k(x)\,g_k(y)\, \right) f(y)\, dy.
		\end{align*}
		Therefore, by Remark \ref{VDelgado}, $T_m$ is a nuclear operator.
		
	\end{proof}
	
	Now, from other viewpoint, we present the following condition for the nuclearity of periodic multilinear operators by studying summability properties of their symbols, these condition can be applied for studying certain symbol classes. The criterion will be presented for periodic Fourier integral operators.
	
	\begin{theorem}\label{sNuclearFIO} Let us consider the real-valued function $\phi:\mathbb{T}^n\times \mathbb{Z}^{nr}\rightarrow \mathbb{R}.$
		Let us consider the Fourier integral operator \begin{equation*}\label{cp2}
		Af(x):=\sum_{\xi\in\mathbb{Z}^{nr}}e^{i\phi(x,\xi_1,\xi_2,\cdots,\xi_r)}a(x,\xi_1,\xi_2,\cdots,\xi_r)(\mathscr{F}_{\mathbb{T}^n}{f}_{1})(\xi_1)\cdots (\mathscr{F}_{\mathbb{T}^n}{f}_{r})(\xi_r) 
		\end{equation*} with symbol satisfying the summability condition
		\begin{equation*}
		\sum_{\xi \in  \mathbb{Z}^{nr} }\Vert a(\cdot,\xi_1,\xi_2,\cdots,\xi_r)\Vert^s_{L^p(\mathbb{T}^n)}<\infty.
		\end{equation*} Then $A$ extends to a $s$-nuclear, $0<s \leq 1,$ operators from $L^{p_1}(\mathbb{T}^n) \times \cdots \times L^{p_r}(\mathbb{T}^n)$ into $L^p(\mathbb{T}^n)$ provided that $1\leq p_j<\infty,$ and $1\leq p\leq \infty.$
	\end{theorem}
	\begin{proof}
		Let us consider the multilinear operator defined by
		\begin{equation}\label{cp}
		Af(x):=\sum_{\xi\in\mathbb{Z}^{nr}}e^{i\phi(x,\xi_1,\xi_2,\cdots,\xi_r)}a(x,\xi_1,\xi_2,\cdots,\xi_r)(\mathscr{F}_{\mathbb{T}^n}{f}_{1})(\xi_1)\cdots (\mathscr{F}_{\mathbb{T}^n}{f}_{r})(\xi_r).
		\end{equation} For every $\xi\in \mathbb{Z}^{nr},$
		define the functions
		\begin{equation*}
		h_\xi(x):=e^{i\phi(x,\xi_1,\xi_2,\cdots,\xi_r)}a(x,\xi_1,\xi_2,\cdots,\xi_r),
		\end{equation*} and the functionals
		\begin{equation*}
		\langle e'_{\xi},f\rangle:=(\mathscr{F}_{\mathbb{T}^n}{f}_{1})(\xi_1)\cdots (\mathscr{F}_{\mathbb{T}^n}{f}_{r})(\xi_r).
		\end{equation*}
		By definition, $A$ extends to a $s$-nuclear multilinear operator from $E=L^{p_1}(\mathbb{T}^n) \times \cdots \times L^{p_r}(\mathbb{T}^n)$ into $F=L^p(\mathbb{T}^n)$ if
		\begin{equation*}
		\sum_{\xi \in  \mathbb{Z}^{nr} }\Vert e_\xi' \Vert^s_{E'}\Vert h_k \Vert^s_{L^p(\mathbb{R}^n)}=\sum_{\xi \in  \mathbb{Z}^{nr} }\Vert e_\xi' \Vert^s_{L^{p_1'}(\mathbb{T}^n) \times \cdots \times L^{p_r'}(\mathbb{T}^n)}\Vert h_k \Vert^s_{L^p(\mathbb{T}^n)}<\infty.
		\end{equation*}
		Taking into account that
		\begin{equation*}
		\Vert h_k \Vert_{L^p(\mathbb{T}^n)}=\Vert a(x,\xi_1,\xi_2,\cdots,\xi_r)\Vert_{L^p(\mathbb{T}^n_x)},
		\end{equation*}
		from the estimate
		\begin{equation*}
		|\langle e'_{\xi},f\rangle|\leq \prod_{j=1}^r\Vert f_j \Vert_{L^{p_j'}}
		\end{equation*}
		we deduce that, $\sup_\xi\Vert e'_\xi \Vert_{E'}\leq 1,$ and the condition
		\begin{equation*}
		\sum_{\xi \in  \mathbb{Z}^{nr} }\Vert a(\cdot,\xi_1,\xi_2,\cdots,\xi_r)\Vert^r_{L^p(\mathbb{T}^n)}<\infty,
		\end{equation*}
		implies the $s$-nuclearity of $A.$  So, we finish the proof.  
	\end{proof}

	\begin{corollary}\label{corollary}
		Let us assume that $a:\mathbb{T}^n\times \mathbb{Z}^{nr}\rightarrow \mathbb{C}$ satisfies estimates of the type,
		\begin{equation*}
		|a(x,\xi_1,\cdots,\xi_r)|\leq C \langle \xi\rangle^{-\varkappa},\,\,x\in\mathbb{T}^n,\,\xi\in\mathbb{Z}^{nr},\varkappa>0.
		\end{equation*}
	\end{corollary}
	If $\varkappa>nr/s,$ the periodic Fourier integral operator $A$  extends to a multilinear $s$-nuclear, $0 <s \leq 1,$  operator from $L^{p_1}(\mathbb{T}^n) \times \cdots \times L^{p_r}(\mathbb{T}^n)$ into $L^p(\mathbb{T}^n)$ for all $1\leq p_j<\infty$ and $1\leq p\leq \infty.$
	
	\begin{proof}
		From Theorem \ref{sNuclearFIO}, the operator $A$ is $s$-nuclear. In fact the series
		\begin{equation*}
		\sum_{\xi \in  \mathbb{Z}^{nr} }\Vert a(\cdot,\xi_1,\xi_2,\cdots,\xi_r)\Vert^s_{L^p(\mathbb{T}^n)}\lesssim \sum_{\xi \in  \mathbb{Z}^{nr} }\langle \xi\rangle^{-\varkappa s} <\infty,
		\end{equation*} converges due to condition  $\varkappa>nr/s.$ The proof is complete.
	\end{proof}
	\begin{example}\label{example1} In order to illustrate the previous conditions, we consider the multilinear Bessel potential. This can be introduced as follows. Consider the periodic multilinear Laplacian denoted by
		\begin{equation*}
		\mathscr{L}:=(\mathcal{L},\cdots,\mathcal{L}),
		\end{equation*} acting on $f=(f_1,\cdots,f_r)\in\mathscr{D}(\mathbb{T}^n)^r$ by

		\begin{align}
		&\mathscr{L}f(x):=(\mathcal{L} f_1(x))\cdots (\mathcal{L}f_r(x)) \\
		&=\sum_{(\xi_1,\cdots,\xi_r)}e^{i2\pi x(\xi_1+\cdots+\xi_r)}|\xi_1|^2\cdots |\xi_r|^2(\mathscr{F}_{\mathbb{T}^n}{f}_{1})(\xi_1)\cdots (\mathscr{F}_{\mathbb{T}^n}{f}_{r})(\xi_r). 
		\end{align} For $r=1,$ we recover the usual periodic Laplacian
		\begin{equation*}
		\mathcal{L}f(x)=-\frac{1}{4\pi^{2}}(\sum_{j=1}^n\partial_{\theta_j}^2)f(x)=\sum_{\xi\in\mathbb{Z}^n}e^{i2\pi x\cdot \xi}|\xi|^2(\mathscr{F}_{\mathbb{T}^n}{f})(\xi).
		\end{equation*} The multilinear Bessel potential of order $\alpha=(\alpha_1,\cdots,\alpha_r)\in \mathbb{N}_0^r,$
		\begin{equation*}
		(I+ \mathscr{L})^{-\frac{\alpha}{2}}:=((I+\mathcal{L})^{ -\frac{\alpha_1}{2}},\cdots,(1+\mathcal{L})^{ -\frac{\alpha_r}{2}  }),
		\end{equation*} 
		can be defined by the Fourier analysis associated to the torus as
		\begin{align}
		&(I+ \mathscr{L})^{-\frac{\alpha}{2}}f(x)=(I+\mathcal{L})^{ -\frac{\alpha_1}{2}}f_1(x)\cdots(1+\mathcal{L})^{ -\frac{\alpha_r}{2}  }f_r(x)\\
		&=\sum_{(\xi_1,\cdots,\xi_r)}e^{i2\pi x(\xi_1+\cdots+\xi_r)}\prod_{j=1}^r(1+|\xi_j|^2)^{-\frac{\alpha_j}{2}}(\mathscr{F}_{\mathbb{T}^n}{f}_{1})(\xi_1)\cdots (\mathscr{F}_{\mathbb{T}^n}{f}_{r})(\xi_r). 
		\end{align} From the estimate  $$a(x,\xi)=\prod_{j=1}^r(1+|\xi_j|^2)^{-\frac{\alpha_j}{2}}\leq \prod_{j=1}^r(1+|\xi_j|^2)^{-\min\limits_{1\leq j\leq r}  \{\frac{\alpha_j}{2}\} }\lesssim \langle \xi\rangle^{ -\min\limits_{1\leq j\leq r}\{{\alpha_j}\}  },$$
		Corollary \ref{corollary} applied to $a(x,\xi)=\prod_{j=1}^r(1+|\xi_j|^2)^{-\frac{\alpha_j}{2}},$ implies that the multilinear Bessel potential $(I+ \mathscr{L})^{-\frac{\alpha}{2}}$ extends to a $s$-nuclear operator   from $L^{p_1}(\mathbb{T}^n) \times \cdots \times L^{p_r}(\mathbb{T}^n)$ into $L^p(\mathbb{T}^n)$ for all $1\leq p_j<\infty$ and $1\leq p\leq \infty$ provided that
		\begin{equation*}
		\varkappa:=\min\limits_{1\leq j\leq r}\{{\alpha_j}\}>{nr}/{s}.
		\end{equation*} This conclusion is sharp, in the sense that if we restrict our analysis to $r=1$ and $p_1=p=2,$ the operator $(I+ \mathcal{L})^{-\frac{\alpha}{2}}$ extends to a $s$-nuclear operator   on $L^{2}(\mathbb{T}^n)$ if and only if $\varkappa:=\alpha>{nr}/{s}={n}/{s}.$ In fact, the class of $s$-nuclear operators on $L^2(\mathbb{T}^n)$ agrees with the Schatten-von Neumann class $S_s(L^2(\mathbb{T}^n))$ of order $s.$ In this case, let us recall that the class $S_r(H)$ of Schatten-von Neumann operators on a Hilbert space $H,$ consists of those compact operators $T$ on $H$ with a system of singular values $\{\lambda_j({T})\}_j:=\textnormal{Spec}(\sqrt{T^*T}),$ satisfying
		\begin{equation*}
		\sum_j \lambda_j({T})^s<\infty.
		\end{equation*} For $T=(I+ \mathcal{L})^{-\frac{\alpha}{2}}$ on $H=L^2(\mathbb{T}^n),$ the system of eigenvalues of the operator $\sqrt{T^*T},$ $\textnormal{Spec}(\sqrt{T^*T})$ is given by
		\begin{equation*}
		\{\lambda_j({T})\}_j=\{(1+|\xi|^2)^{-\frac{\alpha}{2}}:\xi\in\mathbb{Z}^{n}\}.
		\end{equation*} So, for $0<s\leq 1,$  $(I+ \mathcal{L})^{-\frac{\alpha}{2}}$ extends to a $s$-nuclear, $0 <s \leq 1,$ operator   on $L^{2}(\mathbb{T}^n),$ if and only if
		\begin{equation*}
		\sum_j \lambda_j({T})^s=\sum_{\xi\in\mathbb{Z}^{n}} (1+|\xi|^2)^{-\frac{\alpha s}{2}}  <\infty.
		\end{equation*} But, the previous condition holds true if and only $  \alpha>n/s.$
	\end{example}

	\begin{example}
		Now, we consider symbols admitting some type of singularity at the origin. In this general context, let us choose a sequence $\kappa\in L^s(\mathbb{Z}^{nr}).$ Let us consider the symbol
		\begin{equation*}
		a(x,\xi):=\frac{1}{|x|^\rho}\kappa(\xi),\,\,x\in \mathbb{T}^n,\,x\neq 0,\,\xi\in\mathbb{Z}^{nr},\,\,\rho>0 .
		\end{equation*} If we consider the Fourier integral operator associate to $a(\cdot,\cdot),$
		$$ Af(x):=\sum_{\xi\in\mathbb{Z}^{nr}}e^{i\phi(x,\xi_1,\xi_2,\cdots,\xi_r)}\frac{1}{|x|^\rho}\kappa(\xi_1,\cdots,\xi_r)(\mathscr{F}_{\mathbb{T}^n}{f}_{1})(\xi_1)\cdots (\mathscr{F}_{\mathbb{T}^n}{f}_{r})(\xi_r), $$ the condition 
		\begin{equation*}
		0<\rho<n/p,
		\end{equation*}
		implies that the periodic Fourier integral operator $A$  extends to a $s$-nuclear multilinear operator from $L^{p_1}(\mathbb{T}^n) \times \cdots \times L^{p_r}(\mathbb{T}^n)$ into $L^p(\mathbb{T}^n)$ for all $1\leq p_j<\infty$ and $1\leq p\leq \infty.$ In fact, by Theorem \ref{sNuclearFIO}, we only need to verify that
		$$  \sum_{\xi \in  \mathbb{Z}^{nr} }\Vert a(\cdot,\xi_1,\xi_2,\cdots,\xi_r)\Vert^s_{L^p(\mathbb{T}^n)} =\left(\int\limits_{\mathbb{T}^n}\frac{dx}{|x|^{p\cdot\rho }}\right)^{\frac{s}{p}}\sum_{\xi \in  \mathbb{Z}^{nr} }|\kappa(\xi)|^s<\infty. $$ But, this happens only if $0<\rho<n/p.$
	\end{example}
	Sharp conditions for the $s$-nuclearity of pseudo-differential operators on $\mathbb{S}^1$ and $\mathbb{Z}$ were first introduced in Delgado and Wong \cite{DW}. Later, sharp conditions and its applications to H\"ormander classes, Laplacian and sub-Laplacians on arbitrary compact Lie groups were investigated in Delgado and Ruzhansky in \cite{DR,DR1,DR3,DR5,kernelcondition,DRboundedvariable2}. Finally, the approach in the first part in this subsection was adopted to the multilinear case, from  the abstract characterisations of Ghaemi, Jamalpour Birgani, and  Wong \cite{Ghaemi,Ghaemi2} and  Jamalpour Birgani \cite{Majid}. 
	
	\section*{Acknowledgment}
	The authors would like to thank the anonymous referees  for their valuable suggestions  which help us to improve the presentation of this article.
	Vishvesh Kumar thanks the Council of Scientific and Industrial Research, India, for its senior research fellowship. Duv\'an Cardona was partially supported by the Department of Mathematics, Pontificia Universidad Javeriana.
	\bibliographystyle{amsplain}

\end{document}